\documentclass[10pt,a4paper, reqno]{amsart}
\usepackage{amsmath,amsfonts,amssymb,mathrsfs}
\usepackage{mathtools}
\usepackage[left=2cm,right=2cm,top=3cm,bottom=3cm,bindingoffset=0cm]{geometry}
\usepackage{amsthm}
\usepackage{amsbsy}
\usepackage{longtable}
\usepackage{esint}
\usepackage{psfrag}
\usepackage{epsfig}
\usepackage{subfig}
\usepackage{color}		
\usepackage{upgreek}
\usepackage{hyperref}
\usepackage{enumitem}
\providecommand{\doi}[1]{{\hypersetup{urlcolor=black}\href{#1}{#1}}}

\numberwithin{equation}{section}

\newcommand{\R}{\mathbb{R}}
\newcommand{\N}{\mathbb{N}}

\newtheorem{teo}{Theorem}[section]

\newtheorem{rem}[teo]{Remark}
\newtheorem{prop}[teo]{Proposition}
\newtheorem{lemma}[teo]{Lemma}
\newtheorem{cor}[teo]{Corollary}

\renewcommand{\div}{\operatorname{div}}
\renewcommand{\d}[1]{\ensuremath{\operatorname{d}\!{#1}}}

\author{Luca Barbato}
\address[L.\,Barbato]{Mathematical and Physical Sciences for Advanced Materials and Technologies, Scuola Superiore Meridionale, Largo San Marcellino 10, 80138 Napoli, Italy.}
\email{l.barbato@ssmeridionale.it}
\author{Riccardo Molinarolo}
\address[R.\,Molinarolo]{Dipartimento di Informatica, Universit\`a degli Studi di Verona, Strada le Grazie 15, 37134 Verona, Italy}
\email{riccardo.molinarolo@univr.it}

\title[Serrin-type overdetermined problem for the $p$-Laplacian with Robin boundary conditions]{Serrin-type overdetermined problem for the $p$-Laplacian with Robin boundary conditions}

\begin{document}

\begin{abstract}
Let $\Omega$ be an open bounded connected subset of $\R^N$, $N \geq 2$, of class $C^{2,\alpha}$, for $\alpha \in (0,1)$. Let $p \geq 2$ and $\beta>0$. We prove the symmetry of the solution of the $p$-torsion problem with Robin boundary condition subject to the natural overdetermined condition coming from a shape derivative argument and to an extra condition on $\beta$ and on the minimum of the principal curvatures of $\partial\Omega$. The proof is based on some new integral identities, involving the linearized operator of the $p$-Laplacian applied to the standard $P$-function. In passing we prove some other rigidity results in the spirit of Serrin's and Alexandrov's Theorems.
\end{abstract}

\maketitle

\par

\bigskip

\bigskip

\tableofcontents

\section{Introduction}

\subsection{Overview}

In the seminal paper \cite{serrin}, Serrin proved the symmetry of solutions to a wide class of uniformly elliptic overdetermined problems, laying the foundations for a vast field of research. In its simplest formulation, the rigidity result by Serrin can be summarized as follows: if $\Omega \subset \mathbb{R}^N$, $N \ge 2$, is a bounded domain with $C^2$ boundary and $u \in C^2(\overline{\Omega})$ is a solution to 
\begin{equation}\label{intro PbD}
\begin{cases}
\Delta u = N &\text{in } \Omega,
\\
u = 0 &\text{on } \partial\Omega,
\\
u_\nu = c &\text{on } \partial\Omega,
\end{cases}
\end{equation}
where $\nu$ denotes the outer unit normal to $\partial\Omega$, $u_\nu$ is the corresponding normal derivative, and $c$ is a positive constant, then $\Omega$ must be a ball and $u$ is radial. 

Shortly after, Weinberger \cite{w} provided a simpler proof based on integral identities and the introduction of an auxiliary $P$-function, namely
\begin{equation}\label{intro eq P}
        P = \frac{1}{2} |\nabla u|^2 - u \quad\text{on }\overline{\Omega}.
\end{equation}

It is well known that Serrin's proof relies on an adaptation of the \textit{Moving Planes Method}, originally introduced by Alexandrov \cite{a58,a62} to prove the so-called {\it Soap Bubble Theorem}, which asserts that the only compact embedded $(N-1)$-dimensional smooth hypersurface in $\mathbb{R}^N$ with constant mean curvature is the sphere.

The connection between these two fundamental rigidity results has been highlighted by the alternative proof of the result of Alexandrov proposed by Reilly \cite{re77,re82}, which resembles the approach of Weinberger. For further details on this interplay and different proofs of these celebrated results, we refer to two classical surveys: Magnanini \cite{ma2017} and Nitsch and Trombetti \cite{nt}.

It is perhaps impossible to provide an exhaustive account of the vast literature stemming from these two rigidity results, as the field continues to attract significant attention. 
Without any attempts to be exhaustive, we refer to \cite{agbomaz2025, bc, bnst, bh, cgs, cs, dep23, fj, fk, fv, f, fgk, gnn, hp, p,ro} and the references therein. In particular, in a recent breakthrough result,  Figalli and Zhang \cite{fz2025} proved that the Serrin’s Theorem is still valid in the notable mild class of bounded indecomposable sets of finite perimeter (hence, also for Lipschitz domains).

The anisotropic version of the Serrin problem has been studied by Cianchi and Salani \cite{cs} and Bianchini and Ciraolo \cite{bc}, while for the anisotropic Alexandrov theorem we refer He, Li, Ma and Ge \cite{hlmg2009}, De Rosa, Kolasi\'nski, and Santilli \cite{dks2020} and Jia, Wang, Xia, and Zhang \cite{jwxz2023} (see also Morgan \cite{m2005} for $N=2$). 
Furthermore, we mention Ros \cite{ro} for the generalization to constant $r$-mean curvature hypersurface (see also Montier and Ros \cite{mr1991}). 

Although not directly related to the present paper, we briefly point out that the extension of these rigidity results to Riemannian manifolds has attracted considerable attention. For Alexandrov-type theorems, we refer to Brendle \cite{b2013} for warped product manifolds, Fogagnolo and Pinamonti \cite{fp2022} for substatic Riemannian manifolds, and Ciraolo, Roncoroni and Vezzoni \cite{crv2021} for space forms. 

On the other hand, it is well-known that Serrin's result on manifold is, in general, false: we refer to Fall and Minlend \cite{fm2015} for the construction on a compact Riemannian manifold of a counterexample given by a smooth foliation in a neighborhood of a non-degenerate critical point of the scalar curvature. Nevertheless, some rigidity results are available: we mention, for example, Kumaresan and Prajapat \cite{kp98} and Ciraolo and Vezzoni \cite{cv2019} for space form, Farina and Roncoroni \cite{fr} for warped product manifolds and Freitas, Roncoroni, and Santos \cite{frs2024} for manifolds endowed with a closed conformal vector field.

We finally mention, for the sake of completeness, that quantitative stability has been largely investigated for both the aforementioned rigidity results. For the Serrin problem, the first result in that direction has been obtained by Aftalion, Busca, and Reichel \cite{abr1999} with a proof based on a quantitative version of the moving planes method; the result has been then improved by Ciraolo, Magnanini and Vespri \cite{cmv} (see also the inspiring related papers \cite{cms2015,cms2016}). By a careful inspection of the Weinberger’s proof, Brandolini, Nitsch, Salani and Trombetti \cite{bnst08} provided a quantitative estimate of H\"older type (see also the alternative rigidity proof provided in \cite{bnst}). For the Alexandrov's Theorem, the stability has been vastly investigated: we just refer to Ciraolo and Vezzoni \cite{cv18}, Ciraolo and Maggi \cite{cm2017}, Krummel and Maggi \cite{km2017} and references therein (see also the preprint by Figalli and Zhang \cite{fz2026pre}).

A final remark is in order: Magnanini and Poggesi devoted a series of papers to prove sharp quantitative stability estimates for both problems, by exploiting some new integral identities, see \cite{mapo19,mp2020,mp23}. As we will describe, our strategy of proof resembles their approach, see Subsection \ref{subsec strategy} below.

To complete this overview, we also point out that Serrin's problem has been deeply investigated for different operators and under different boundary conditions. Fully nonlinear operators of non-divergence form, such as $k$-Hessian operators, have been analyzed by Brandolini, Nitsch, Salani, and Trombetti \cite{bnst}, by Brandolini, Gavitone, Nitsch, Trombetti \cite{bgnt}, and by Silvestre and Sirakov \cite{ss}. 

In particular, the $p$-Laplacian operator has been studied by Garofalo and Lewis \cite{gl} by a Weinberger’s approach, by Brock and Henrot \cite{bh} for $p\geq 2$ via Steiner symmetrization, and by Damascelli and Pacella \cite{dp} for $1< p< 2$ by moving plane method. In the latter regime, we also mention the alternative approach by integral identities by Colasuonno and Ferrari \cite{CoFe20} from which our analysis stems (see also the analogous result on complete noncompact Riemannian manifolds provided by Ruan, Huang, and Chen \cite{RuHuCh24}). Finally, for the $\infty$-Laplacian we refer to Buttazzo and Kawohl \cite{bk2011}. 

On the other hand, concerning different types of boundary conditions, we mention the analysis performed in \cite{mmp} and the introduction of the so-called \textit{Reverse Serrin Problem}, where essentially the non-zero constant Neumann boundary condition is imposed and the Dirichlet trace is allowed to oscillate. For non constant boundary value data, see Dom\'inguez-V\'azquez, Enciso, and Peralta-Salas \cite{dep23}. 

Our result is inspired by a result of Gavitone and the second author \cite{gamo26}. There, the authors considered an overdetermined problem for the torsion function subject to Robin boundary conditions, namely for $\Omega \subset \R^N$ a smooth bounded domain and $\beta > 0$, let $u \in C^2(\overline{\Omega})$ be a solution
\begin{equation} \label{intro eq: problem_P p=2}
    \begin{cases}
    \Delta u = N & \text{in } \Omega, \\
    u_\nu  + \beta u = 0 & \text{on } \partial\Omega,
    \end{cases}
\end{equation}
subject to the overdetermined condition
\begin{equation}\label{intro eq overdetermined cond p=2}
    |\nabla u|^2 + 2Nu - 2(u_\nu)^2 + (N-1) \beta u^2 \mathrm{M}_{\partial\Omega} = \Tilde{C} \quad \text{in } \Omega,
\end{equation}
where $\mathrm{M}_{\partial\Omega}$ is the mean curvature of $\partial \Omega$ and $\Tilde{C}$ is a constant. The overdetermined condition \eqref{intro eq overdetermined cond p=2} comes from a shape derivative argument under volume constrain contained in \cite{bw} for the energy associated to \eqref{intro eq: problem_P p=2}, namely 
\begin{equation*}
    \mathcal{E}_{\beta}(u) = \int_{\Omega} |\nabla u|^2 \,\d{x} + 2 \int_{\Omega} N u \,\d x + \beta  \int_{\partial\Omega} u^2 \, \d{}\mathcal{H}^{N-1},
\end{equation*}
(see also the conjecture in \cite[Thm. 4.1 \& Ex. 4.1]{bw}). More precisely, a domain $\Omega$ is a critical point for the energy $ \mathcal{E}_{\beta}(u)$ if and only if \eqref{intro eq overdetermined cond p=2} holds. Analogously to the Serrin's result \cite{serrin}, in \cite{gamo26} the authors took a first step toward resolving that conjecture by proving that, if $\Tilde{C} = \Tilde{C}_0$, where $\Tilde{C}_0 = -R^{2} - \frac{R}{\beta} (N + 1)$, then $\Omega$ is a ball and $u$ is radial. Notice that, formally for $\beta \to +\infty$, \eqref{intro eq: problem_P p=2}-\eqref{intro eq overdetermined cond p=2} converges to \eqref{intro PbD}, therefore, a posteriori, it is not surprising that the same $P$-function given by \eqref{intro eq P} works for both problems.

The goal of this paper is to generalize to the $p$-Laplacian for $p \geq 2$ and with Robin boundary conditions the results contained in \cite{gamo26}, proving new integral identities in the spirit of \cite{mapo19,mp2020,mp23}. 

\subsection{Main result}

In this paper we consider the Serrin’s problem for
the $p$-torsion function with Robin boundary conditions. More precisely, let $\Omega$ be an open bounded connected subset of $\R^N$ with $C^{2,\alpha}$ boundary, for $\alpha \in (0,1)$, and let us fix once for all two parameters
\[
\beta >  0 \text{ and } p \geq 2.
\]
Let us consider the problem
\begin{equation} \label{intro eq:problem_p}
    \begin{cases}
    \Delta_p u = N & \text{in } \Omega, \\
    |\nabla u|^{p-2} u_\nu  + \beta |u|^{p-2}u = 0 & \text{on } \partial\Omega,
    \end{cases}
\end{equation}
where $\Delta_p u = \div(|\nabla u|^{p-2}\nabla u)$ is the $p$-Laplacian operator, and $\nu$ is the outer unit normal vector field on $\partial\Omega$. 

A shape derivative argument under volume constraint for the energy associated to problem \eqref{intro eq:problem_p} (cf. Appendix \ref{sec: appendix shape}), yields to consider the following overdetermined condition:
\begin{equation}\label{intro eq overdetermined cond p}
   |\nabla u|^p + pNu - p |\nabla u|^{p-2} (u_\nu)^2 + \beta(N-1) |u|^p \mathrm{M}_{\partial\Omega} = C  \quad\text{on } \partial\Omega,
\end{equation}
for a constant $C\in \R$ (cf. \eqref{intro eq overdetermined cond p=2}). In particular, one may investigate if an analogous result to \cite[Thm. 1]{gamo26} holds for the $p$-Laplacian operator under suitable assumptions.
Define the constants
\begin{equation}\label{intro eq R}
        R = \frac{N|\Omega|}{|\partial\Omega|}, 
\end{equation}
and 
\begin{equation}\label{intro eq C0}
C_0 = (1-p)R^{\frac{p}{p-1}} - \left( \frac{R}{\beta} \right)^{\frac{1}{p-1}} \big( (p-1)N + 1 \big).    
\end{equation}
If $\kappa_i$ for $i=1, \dots,N-1$ denotes the $i^{\text{th}}$ principal curvature of $\partial \Omega$, set
\begin{equation}\label{intro k min}
    \kappa_{\min} = \min_{\partial \Omega} \min_{i=1,\dots,N-1}  \kappa_i(x).
\end{equation}
Denote by $\mathrm{M}_{\partial\Omega}$ the mean curvature of $\partial \Omega$. Our main result is the following.

\begin{teo}\label{intro main thm rigidity serrin}
    Let $\Omega$ be an open bounded connected subset of $\R^N$ with $C^{2,\alpha}$ boundary, for $\alpha \in (0,1)$. Let $R, C_0$ and $\kappa_{\min}$ be given by \eqref{intro eq R},\eqref{intro eq C0} and \eqref{intro k min}, respectively. Assume that
    \begin{equation}\label{intro eq condition kmin beta}
        \kappa_{\min} +  \frac{3p-4}{2} \beta^{\frac{1}{p-1}} > 0.
    \end{equation}
    Let $u \in C^{1,\gamma}(\overline{\Omega})$, for $\gamma \in (0,1)$, be a solution to
    \begin{equation}\label{intro eq: overdetermined problem}
    \begin{cases}
    \Delta_p u = N & \text{in } \Omega, \\
    |\nabla u|^{p-2} u_\nu  + \beta |u|^{p-2}u = 0 & \text{on } \partial\Omega, \\
    |\nabla u|^p + pNu - p |\nabla u|^{p-2} (u_\nu)^2 + \beta(N-1) |u|^p \mathrm{M}_{\partial\Omega} = C & \text{on } \partial\Omega,
    \end{cases}
    \end{equation}
    with $C \geq C_0$ any constant. Then, $\Omega$ is a ball of radius $R$ and $u$ is radially symmetric.
\end{teo}

\subsection{Strategy of the proof}\label{subsec strategy}

The first step consists in proving a \textit{Fundamental Identity} regarding the linearized operator $\mathcal{L}_u$ (cf. \eqref{eq: L(phi)_function}) of the $p$-Laplacian operator applied to the $P$-function given by
\begin{equation}\label{intro eq:P_function p}
    P(u) = \frac{2(p-1)}{p}|\nabla u|^p - 2u \quad \text{on } \overline{\Omega},
\end{equation}
where $u \in C^{1,\gamma}(\overline{\Omega})$ is a solution of \eqref{intro eq:problem_p}. It simply reads (see Theorem \ref{teo:fundamental}):
\begin{equation}\label{intro eq: fundamental identity}
    \begin{aligned}
    \frac{1}{2(p-1)}\int_\Omega \mathcal{L}_u(P) \d{x} +\int_{\partial\Omega} |\nabla u|^{2p-4} \mathrm{I\!I}(\nabla_\top u,\nabla_\top u) \d{}\mathcal{H}^{N-1}+ 2(p-1) \beta\int_{\partial\Omega} |\nabla u|^{p-2}|u|^{p-2}|\nabla_\top u|^2 \d{}\mathcal{H}^{N-1} \\= (N-1) \int_{\partial\Omega} |\nabla u|^{p-2} u_\nu \d{}\mathcal{H}^{N-1}
    - (N-1) \int_{\partial\Omega} |\nabla u|^{2p-4} (u_\nu)^2 \mathrm{M}_{\partial\Omega} \d{}\mathcal{H}^{N-1},
    \end{aligned}
\end{equation}
where $\mathrm{I\!I}(\cdot,\cdot)$ denotes the second fundamental form on $\partial\Omega$ and $\nabla_\top u$ the tangential gradient of $u$. The proof is essentially based on integration by parts, Reilly's identity for the $p$-Laplacian operator and a suitable relation coming from the tangential derivation of the Robin boundary condition, see Lemma \ref{lem:tangential} below.
\\
Then, we rewrite \eqref{intro eq: fundamental identity} in order to make the deficit of the overdetermined condition appear, namely the quantity
\begin{equation*}
    |\nabla u|^p + pNu - p |\nabla u|^{p-2} (u_\nu)^2 + \beta(N-1) |u|^p \mathrm{M}_{\partial\Omega} - C  \quad\text{on } \partial\Omega.
\end{equation*}
We obtain a new general identity, which we believe has his own independent interest:
\begin{equation}\label{intro eq: fundamental identity for serrin}
    \begin{aligned}
    &\frac{1}{2(p-1)}\int_\Omega \mathcal{L}_u(P) \d{x} +\int_{\partial\Omega} |\nabla u|^{2p-4} \mathrm{I\!I}(\nabla_\top u,\nabla_\top u) \d{}\mathcal{H}^{N-1} 
    \\
    &+ \frac{3p-4}{2} \beta \int_{\partial\Omega} |\nabla u|^{p-2}|u|^{p-2}|\nabla_\top u|^2 \d{}\mathcal{H}^{N-1} 
    + (p-1) \beta^{\frac{1}{p-1}} \int_{\partial\Omega} \Big(|\nabla u|^{p-2} u_\nu - R \Big)^2  \d{}\mathcal{H}^{N-1}  
    \\
    &
    + (p-1) \beta^{\frac{1}{p-1}}  \int_{\partial\Omega}  \left( \frac{1}{2} \frac{p-2}{p-1} \left(\frac{u_\nu}{|\nabla u|}\right)^{-\frac{p}{p-1}} \left(1- \left(\frac{u_\nu}{|\nabla u|}\right)^2 \right) - \left(1- \left(\frac{u_\nu}{|\nabla u|}\right)^{\frac{p-2}{p-1}} \right) \right) (|\nabla u|^{p-2} u_\nu)^2 \d{}\mathcal{H}^{N-1} 
    \\
    &
    + \left( (p-1) \beta^{\frac{1}{p-1}} + (p-1)N + 1 \right) R^{\frac{1}{p-1}}  \int_{\partial\Omega} \Big( R^{\frac{p-2}{p-1}} - (|\nabla u|^{p-2} u_\nu)^{\frac{p-2}{p-1}} \Big)  \d{}\mathcal{H}^{N-1}
    \\
    &
    = - \beta \int_{\partial\Omega} \Big( |\nabla u|^p + pNu - p |\nabla u|^{p-2} (u_\nu)^2 + \beta(N-1) |u|^p \mathrm{M}_{\partial\Omega}  - C_0 \Big) |u|^{p-2} \d{}\mathcal{H}^{N-1},  
    \end{aligned}
\end{equation}
for any function $u \in C^{1,\gamma}(\overline{\Omega})$ solution of \eqref{intro eq:problem_p} (see Theorem \ref{teo: identity serrin}).

Then, we devote a series of lemmas to prove the sign of the left-hand side of \eqref{intro eq: fundamental identity for serrin} under the hypothesis \eqref{intro eq condition kmin beta}. 
Analogously to the case $p=2$, in which the related $P$-function is well-known to be subharmonic, in Lemma \ref{lemma positivity L(P)}  we prove that if $P$ is given by \eqref{intro eq:P_function p}, then
\[
\mathcal{L}_u(P) \geq 0 \quad \text{and} \quad \mathcal{L}_u(P)=0 \iff u \text{ is radial}.
\]
We also point out the novelty of both the convexity inequality provided by Lemma \ref{lem:jensen ineq} and the ``ad hoc'' inequality proven in Lemma \ref{lem:rho ineq}: these two results, heavily relying on the condition $p\geq 2$, dictate the sign of the fifth and sixth terms of the left-hand side of \eqref{intro eq: fundamental identity for serrin}. Notice in particular that both integrals are null for $p=2$.

Identity \ref{intro eq: fundamental identity for serrin} allows us to prove a quite general Serrin-type rigidity result (see Theorem \ref{thm rigidity serrin integral}) in which we only assume \eqref{intro eq condition kmin beta} and
\begin{equation}\label{intro eq: integral condition rigidity}
    \int_{\partial\Omega} \Big( |\nabla u|^p + pNu - p |\nabla u|^{p-2} (u_\nu)^2 + \beta(N-1) |u|^p \mathrm{M}_{\partial\Omega}  - C_0 \Big) |u|^{p-2} \d{}\mathcal{H}^{N-1} \geq 0.
\end{equation}
It is then clear that Theorem \ref{intro main thm rigidity serrin} follows at once as a corollary from \eqref{intro eq: integral condition rigidity}, once the overdetermined condition in \eqref{intro eq: overdetermined problem} for $C \geq C_0$ is in force. Moreover, since $\beta > 0$, if one assumes $\Omega$ is convex, then \eqref{intro eq condition kmin beta} is clearly fulfilled,  and again rigidity follows simply assuming \eqref{intro eq: integral condition rigidity} (see Corollary \ref{cor rigidity serrin integral}).

In passing, from \eqref{intro eq: fundamental identity}, we also deduce a different identity, namely
\begin{equation}\label{intro eq alexandrov}
    \begin{aligned}
    \frac{1}{2(p-1)}\int_\Omega \mathcal{L}_u(P) \d{x} 
    +\int_{\partial\Omega} |\nabla u|^{2p-4} \mathrm{I\!I}(\nabla_\top u,\nabla_\top u) \d{}\mathcal{H}^{N-1}
    + 2(p-1)\beta \int_{\partial\Omega} |\nabla u|^{p-2}|u|^{p-2}|\nabla_\top u|^2 \d{}\mathcal{H}^{N-1} 
    \\
    + (N-1) \int_{\partial\Omega} \Big(|\nabla u|^{p-2} u_\nu - R \Big)^2  \mathrm{M}_{0} \d{}\mathcal{H}^{N-1} 
    =
    - (N-1) \int_{\partial\Omega} |\nabla u|^{2p-4} (u_\nu)^2 (\mathrm{M}_{\partial\Omega} - \mathrm{M}_{0} ) \d{}\mathcal{H}^{N-1}.  
    \end{aligned}
\end{equation}
Identity \eqref{intro eq alexandrov} provides an alternative proof of the Alexandrov's Theorem, using a solution $u$ of \eqref{intro eq:problem_p} (see Theorems \ref{teo: identity soap} \& \ref{thm alexandrov}).

A final remark is in force: a fundamental step in our analysis is given by the Hold\"er regularity of the weak solution $u$ of problem \eqref{intro eq:P_function p}. In particular, we took advantage of the $C^{1,\gamma}$ regularity of $u$, of the constant sign of $u$ in $\overline{\Omega}$, and of the absence of critical points of $u$ on $\partial\Omega$, which, essentially by standard elliptic regularity, implies that $u$ is of class $C^{2,\min\{\alpha,\gamma\}}$ in an $\varepsilon$-neighborhood of $\partial\Omega$. Thus, $\nabla^2 u \in C^0(\partial\Omega)$ and the set $Z$ of critical point of $u$ in $\Omega$ has zero Lebesgue measure. In particular, this justifies the integration of $\mathcal{L}_u(P)$ over $\Omega$.

\subsection{Structure of the paper}
The paper is organized as follows. In Section \ref{sec: notation} we introduce the notation and recall some basic facts from differential geometry, Divergence Theorem on hypersurfaces, Reilly's identity and Newton's inequality. In Section \ref{sec: properties and regularity} we summarize standard properties and regularity results for the weak solution of the $p$-torsion problem with Robin boundary conditions and enhanced regularity in a neighborhood of the boundary and we recall the structure of radial solutions on a ball. In Section \ref{sec: integral identities} we present some key integral identities along with the fundamental tangential derivation of the Robin boundary condition. Section \ref{sec: rigidity} is devoted to prove some rigidity results and Theorem \ref{intro main thm rigidity serrin}. Finally, two Appendices conclude the paper: Appendix \ref{sec: appendix reg} contains the proof of Theorem \ref{thm regularity C2} of the improved regularity of the weak solution in a neighborhood of the boundary, while Appendix \ref{sec: appendix shape} is devoted to describe a shape derivative argument under volume constraint leading to the natural overdetermined condition for the $p$-torsion problem with Robin boundary conditions.
 
\section{Notation and preliminaries} \label{sec: notation}

We recall some standard notation used throughout the paper. 

Let $N \in \N $, with $N\ge 2$. We denote by $x = (x_1,\dots,x_N)$ a point in $\R^N$ and we endowed the Euclidean space $\R^N$ with the standard scalar product denoted by $\langle \cdot,\cdot \rangle$. Let $\Omega \subset \R^N$ be an open and bounded domain with a boundary of class $C^{1,1}$. We denote by $\mathrm{M}_{\partial\Omega}$ the mean curvature of $\partial \Omega$, namely
\[
\mathrm{M}_{\partial\Omega}=\frac{1}{N-1} \sum_{i=1}^{N-1} \kappa_i,
\]
where $\kappa_i$ for $i=1, \dots,N-1$ denotes the $i^{\text{th}}$ principal curvature of $\partial \Omega$. Here the convention is that $\kappa_i$ are oriented so that convex sets have positive curvatures.

Let $\nu$  be the unit outer normal to the boundary. For a function $f\in C^1(\overline{\Omega})$, we denote by $\nabla f = \left(\frac{\partial f}{\partial x_1},\dots,\frac{\partial f}{\partial x_N} \right)$ the gradient of $f$, while we denote by $\nabla_\top f$ the tangential gradient of the function $f$ on $\partial \Omega$, given by
\begin{equation*}
\nabla_\top f= \nabla f- f_\nu \, \nu, \quad \text{where } f_\nu = \langle \nabla f, \nu \rangle.
\end{equation*}
Moreover, for any smooth vector field $\mathbf{u} \colon \Omega \to \R^N$, the tangential divergence $\div_{\top} \mathbf{u}$ is given by
\begin{equation*}
\div_{\top} \mathbf{u} =  \div \mathbf{u} - \langle \nabla \mathbf{u} \,\nu , \nu\rangle, 
\end{equation*}
where $\nabla \mathbf{u} \in \R^{N \times N}$. In particular, as standard, for $u \in C^{2}(\partial \Omega)$ we set $\Delta_{\top} u := \div_{\top} \nabla_{\top} u$.
We retain classical notations and definitions for standard Sobolev and Hold\"er spaces (see \cite{gitr83}). 

We recall the statement of the Divergence Theorem on $\partial \Omega$ (see \cite{giusti84,li12,hepi18}).

\begin{teo}[Divergence Theorem on $\partial \Omega$]\label{thm: divergence part Om}
    Let $\Omega$ be an open, bounded subset of $\R^N$ with $C^{1,1}$ boundary. Let $f\in C^1(\partial \Omega) $ and $\mathbf{u} \in C^{1}(\partial \Omega, \R^N)$. Then the following integration by parts formula holds:
    \begin{equation*}
    \int_{\partial \Omega} f \div_{\top} \mathbf{u} \, \d{}\mathcal{H}^{N-1} =-\int_{\partial \Omega} \langle \mathbf{u},\nabla_{\top} f\rangle \, \d{}\mathcal{H}^{N-1} + (N-1)\int_{\partial \Omega} f \langle \mathbf{u},\nu\rangle \mathrm{M}_{\partial\Omega} \, \d{}\mathcal{H}^{N-1}.
    \end{equation*} 
    In particular, for $f\in C^1(\partial \Omega)$ and $u \in C^{2}(\partial \Omega)$, the following formula holds:
    \begin{equation*}
        \int_{\partial \Omega} f \Delta_{\top} u \, \d{}\mathcal{H}^{N-1} =-\int_{\partial \Omega} \langle \nabla_{\top}u,\nabla_{\top} f\rangle \,\d{}\mathcal{H}^{N-1}.
    \end{equation*} 
\end{teo}

Furthermore, we recall that the mean curvature can be introduced also by the relation 
\begin{equation*}
\div_{\top} \nu = (N-1) \mathrm{M}_{\partial\Omega}.    
\end{equation*} 
In particular, we have $\mathrm{M}_{\partial B_R} = \frac{1}{R}$, where $B_R$ denotes the ball of radius $R$ in $\R^N$ centered at the origin.

Finally, we recall the well-known differential formula, the Reilly's identity for the $p$-Laplacian for a smooth function on the closed smooth hypersurface $\partial \Omega$ (see \cite{hepi18}):
\begin{equation}\label{eq: reilly formula}
\begin{split}
    \Delta_p u & = |\nabla u|^{p-2} \Delta u + (p-2) |\nabla u|^{p-2} \frac{\langle(\nabla^2u)\nabla u,\nabla u\rangle}{|\nabla u|^{2}}
    \\
    & = |\nabla u|^{p-2} \left( \langle(\nabla^2u)\nu,\nu\rangle + \Delta_\top u  + (N-1) u_\nu \mathrm{M}_{\partial\Omega} + (p-2)\frac{\langle(\nabla^2u)\nabla u,\nabla u\rangle}{|\nabla u|^{2}} \right) \quad \text{on } \partial\Omega.
\end{split}
\end{equation}

Now we recall the following well-known algebraic result, an important tool for rigidity results (for the proof of this result we refer, for instance, to \cite[Prop. 2.4]{CoFe20}). We denote
by $\|\cdot\|$ the Frobenius matrix norm induced by the corresponding Frobenius scalar product $\langle \cdot , \cdot \rangle_{\mathrm{Fr}}$ of matrices of $\R^{N\times N}$.

\begin{prop}[Newton’s inequality] \label{prop newton}
    Let $N \in \N$ and $A \in \R^{N \times N}$, then 
    \begin{equation*}
        \|A\|^2 \geq \frac{(\mathrm{tr}(A))^2}{N},
    \end{equation*}
    where denotes $\mathrm{tr}(\cdot)$ the trace of a matrix. Furthermore, the equality holds if and only if $A = k I_N$ for some constant $k \in \R$.
\end{prop}

\section{Properties and regularity}\label{sec: properties and regularity}

Fix $\Omega$ an open bounded connected subset of $\R^N$ with $C^{2,\alpha}$ boundary, for $\alpha \in (0,1)$. As already mentioned, for a real positive parameter $\beta > 0$ and $p\geq2$, we consider the following boundary value problem:
\begin{equation} \label{eq:problem_p}
    \begin{cases}
    \Delta_p u = N & \text{in } \Omega, \\
    |\nabla u|^{p-2} u_\nu + \beta |u|^{p-2}u = 0 & \text{on } \partial\Omega.
    \end{cases}
\end{equation}
As customary, $u\in W^{1,p}(\Omega)$ is a weak solution of \eqref{eq:problem_p} if and only if for any test function $\varphi \in  W^{1,p}(\Omega)$ it holds
\begin{equation*}
    \int_\Omega |\nabla u|^{p-2} \langle \nabla u, \nabla \varphi \rangle \,\d{x} +\beta \int_{\partial\Omega}|u|^{p-2}u\,\varphi\,\d{\mathcal H^{n-1}}
    =-N\int_\Omega \varphi \,\d{x}.
\end{equation*}
It is well-known that, for the $p$-Laplacian operator, Hopf's Lemma and Maximum Principle hold (see \cite[Chapter 5]{PucciSerrin07}). For sake of completeness, we start by recalling  proposition on known properties of the weak solutions of problem \ref{eq:problem_p}. 

\begin{prop}\label{prop properties u}
There exists a unique weak solution $u$ of \eqref{eq:problem_p}. Moreover,
\begin{equation*}
    u \in C^{1,\gamma}(\overline{\Omega}) \quad \text{and} \quad u \in C^{2,\alpha}(\Omega \setminus Z),
\end{equation*}
for some $\gamma \in (0,1)$, where $Z := \{x \in \Omega \colon |\nabla u| = 0\}$.
Furthermore, the set $Z$ has zero Lebesgue measure,
\[
u<0 \quad \text{in } \overline{\Omega} \quad \text{and} \quad
\nabla u \neq 0 \quad \text{on } \partial\Omega.
\]
\end{prop}

\begin{proof}
The global regularity of the unique weak solution $u$ of problem \eqref{eq:problem_p} follows by \cite[Theorem 2]{Lieberman88}. For the measure of the set $Z$ we refer to \cite[Corollary 1.1]{Lou2008}. Standard elliptic regularity theory yields the smoothness of $u$ on $\Omega \setminus Z$, see \cite[Lemma 3.1]{fgk}.

Testing the weak formulation of \eqref{eq:problem_p} with $\varphi= u_+ := \max\{u,0\}$ yields
\[
\int_\Omega |\nabla u|^{p-2} \langle\nabla u,  \nabla u_+ \rangle\,\d{x}
+\beta\int_{\partial\Omega}|u|^{p-2}u\,u_+\,\d{\mathcal H^{n-1}}
=-N\int_\Omega u_+\,\d{x}.
\]
Since $\nabla u_+=\nabla u$ a.e.\ on $\{u>0\}$ and $\nabla u_+=0$ a.e.\ on $\{u\le 0\}$, we obtain
\[
\int_\Omega |\nabla u_+|^p\,\d{x}
+\beta\int_{\partial\Omega}u_+^p\,\d{\mathcal H^{n-1}}
=-N\int_\Omega u_+\,\d{x}.
\]
The left-hand side is nonnegative, while the right-hand side is nonpositive. Hence both sides must vanish, and therefore $u_+\equiv 0$ in $\overline{\Omega}$.
Thus
\[
u\le 0 \qquad \text{in }\overline{\Omega}.
\]

We now show that $u<0$ on $\partial\Omega$. Assume by contradiction that there exists $x_0\in\partial\Omega$ such that
\[
u(x_0)=0.
\]
Since $u\le 0$ in $\overline{\Omega}$, the point $x_0$ is a global maximum point of $u$. Moreover, $u$ is not constant, because otherwise $\Delta_p u\equiv 0$ in $\Omega$, 
contradicting \eqref{eq:problem_p}. Hence, by the strong maximum principle and the Hopf boundary lemma (cf. \cite[Chapter 5]{PucciSerrin07}), we get
\[
u_\nu(x_0) > 0.
\]
In particular, $\nabla u(x_0)\neq 0$, so that
\[
|\nabla u(x_0)|^{p-2} u_\nu(x_0)>0.
\]
On the other hand, evaluating the Robin condition at $x_0$ and using $u(x_0)=0$, we get $|\nabla u(x_0)|^{p-2} u_\nu(x_0)=0$, that is a contradiction. Therefore,
\[
u<0 \quad \text{on }\partial\Omega.
\]
Finally, the above sign relation along with $\beta >0$ and the Robin boundary condition yield $\nabla u \neq 0$ on $\partial\Omega$.
\end{proof}

By Proposition \ref{prop properties u}, it follows that $\nabla u \neq 0$ in a suitable neighborhood of $\partial\Omega$. In particular, in this neighborhood, the $p$-Laplacian is uniformly elliptic. Therefore, one expects standard elliptic regularity theory to yield enhanced regularity of $u$ in a neighborhood of $\partial\Omega$ up to the boundary. We formalize that in the following result (the proof is postponed to Appendix \ref{sec: appendix reg}). 

\begin{teo}\label{thm regularity C2}
Let $u\in C^{1,\gamma}(\overline{\Omega})$ be the solution of \eqref{eq:problem_p}. Then, there exists $\varepsilon>0$ such that
\[
    u\in C^{2,\min\{\alpha,\gamma \}}(\overline{\Omega_\varepsilon}), 
\]
where $\Omega_\varepsilon=\{x\in\Omega: \mathrm{dist}(x,\partial\Omega)<\varepsilon\}$.
\end{teo}

\begin{rem}[Value of the constants $R$ and $C_0$]
    We first notice that by Divergence Theorem and by \eqref{eq:problem_p} we get
    \begin{equation*}
        N|\Omega| = \int_\Omega \Delta_p u \d{x} = \int_{\partial\Omega} |\nabla u|^{p-2} u_\nu \d{}\mathcal{H}^{N-1} = |\partial\Omega|R,
    \end{equation*}
    if one sets $R=\frac{N|\Omega|}{|\partial\Omega|}$ (cf. \eqref{intro eq R}).

    It is well known that the problem \eqref{eq:problem_p} admits a unique radial solution when $\Omega = B_R(z)$, with $z \in \R^N$. We denote this radial solution by $q^z$ and it has the explicit form:
    \begin{equation*}
        q^z(x) = \frac{p-1}{p} \left(|x-z|^{\frac{p}{p-1}}-R^\frac{p}{p-1}\right) -\left(\frac{R}{\beta}\right)^\frac{1}{p-1} \quad \text{for } x \in \overline{B_R(z)}.
    \end{equation*}
    In particular, on the boundary $\partial B_R(z)$ where $|x-z| = R$, the outer normal is given by $\nu(x) = \frac{x-z}{R}$, and we have
    \begin{equation*}
        \nabla q = R^{\frac{1}{p-1}} \frac{x-z}{R}, \quad |\nabla q| = R^{\frac{1}{p-1}}, \quad q=-\left( \frac{R}{\beta} \right)^{\frac{1}{p-1}},  \quad \text{and} \quad q_\nu = R^{\frac{1}{p-1}} \quad \text{on } \partial B_R(z).
    \end{equation*}
    Moreover, the mean curvature is $\mathrm{M}_{\partial B_R(z)} = \frac{1}{R}$. Hence we can directly compute the constant value on the overdetermined condition \eqref{intro eq overdetermined cond p} for the radial solution $q$ on $B_R(z)$, obtaining:
    \begin{equation*}
        \begin{aligned}
        C_0 
        &= |\nabla q|^p + pNq - p|\nabla q|^{p-2} (q_\nu)^2 + (N-1)\beta|q|^p \mathrm{M}_{\partial B_R(z)} \\
        &= \left( R^{\frac{1}{p-1}} \right)^p - pN \left( \frac{R}{\beta} \right)^{\frac{1}{p-1}} - p \left( R^{\frac{1}{p-1}} \right)^{p-2} \left( R^{\frac{1}{p-1}} \right)^2 + (N-1)\beta \left( \frac{R}{\beta} \right)^{\frac{p}{p-1}} \frac{1}{R} \\
        &= R^{\frac{p}{p-1}} - pN \left( \frac{R}{\beta} \right)^{\frac{1}{p-1}} - p R^{\frac{p}{p-1}} + (N-1) \left( \frac{R}{\beta} \right)^{\frac{1}{p-1}}  \\
        &= (1-p)R^{\frac{p}{p-1}} - \left( \frac{R}{\beta} \right)^{\frac{1}{p-1}} \big[ pN - (N-1) \big].
        \end{aligned}
    \end{equation*}
    which yields \eqref{intro eq C0}. In particular, we notice that ``a fortiori''
    \begin{equation*}
        N|B_R| = \int_{\partial B_R} |\nabla q|^{p-2} q_\nu \d{}\mathcal{H}^{N-1} = |\partial B_R|R^\frac{p-2}{p-1}R^\frac{1}{p-1}=|\partial B_R|R.
    \end{equation*}
\end{rem}

\section{Integral identities}\label{sec: integral identities}
In this section, we fix once of all 
\[
\text{an open bounded connected subset $\Omega$ of } \R^N \text{ with } C^{2,\alpha} \text{ boundary, for } \alpha \in (0,1).
\]
Let $u \in C^{1,\gamma}(\overline{\Omega})$ be the solution of \eqref{intro eq:problem_p} provided by Proposition \ref{prop properties u}. We define the $P$-function as
\begin{equation}\label{eq:P_function}
    P(u) = \frac{2(p-1)}{p}|\nabla u|^p - 2u \quad \text{on } \overline{\Omega}.
\end{equation}
We consider the linearized operator associated with the $p$-Laplacian, denoted by $\mathcal{L}_u$, acting on a smooth test function $\phi$ by
\begin{equation}\label{eq: L(phi)_function}
    \mathcal{L}_u(\phi) = \div\big( |\nabla u|^{p-2} A_p(\nabla u) \nabla \phi \big) \quad \text{in } \Omega,
\end{equation}
where the matrix $A_p(\xi)$ is given by
\begin{equation}\label{eq:matrix_A}
    A_p(\xi) = I_N + (p-2)\frac{\xi \otimes \xi}{|\xi|^2} \quad \text{for all } \xi \in \R^N \setminus \{0\}.
\end{equation}
In particular, it is well-known that, applying $\mathcal{L}_u$ to $u$, we retrieve the $p$-Laplacian structure, namely
\begin{equation}\label{eq: L(u)_u}
\begin{aligned}
    \mathcal{L}_u(u) &= \div\big( |\nabla u|^{p-2} A_p(\nabla u) \nabla u \big) = \div\left( |\nabla u|^{p-2} \left( I_N + (p-2)\frac{\nabla u \otimes \nabla u}{|\nabla u|^2} \right) \nabla u \right) \\
    &= \div\big( |\nabla u|^{p-2} ( \nabla u + (p-2)\nabla u ) \big) = (p-1)\Delta_p u.
\end{aligned}
\end{equation}
Moreover, for $p=2$, then $A_2 = I_N$ and, consequently, $\mathcal{L}_u(P) = \Delta P$, where in this case $P$ is given by \eqref{intro eq P}.                  

We start by the following generalization of \cite[Lemma 1]{gamo26} providing a formula for the derivative of the Robin boundary condition. Notice that for $p=2$ we recast exactly \cite[Eq.\,(3.1)]{gamo26}

\begin{lemma}[Tangential derivative of the boundary condition] \label{lem:tangential}
    Let $u \in C^{1,\gamma}(\overline{\Omega})$ be the solution to \eqref{eq:problem_p}. Then, we have
    \begin{equation*}
        \begin{aligned}
        |\nabla u|^{p-2}\langle(\nabla^2u)\nabla u,\nu\rangle = & \, |\nabla u|^{p-2} u_\nu \langle(\nabla^2u)\nu,\nu\rangle
        \\&-(p-2)|\nabla u|^{p-4} u_\nu \langle(\nabla^2u)\nabla u,\nabla_\top u\rangle
        \\&-|\nabla u|^{p-2} \mathrm{I\!I}(\nabla_\top u,\nabla_\top u)
        -(p-1)\beta|u|^{p-2} |\nabla_\top u|^{2} \qquad \text{on }\partial\Omega.
        \end{aligned}
    \end{equation*}
    where $\mathrm{I\!I}$ denotes the second fundamental form of $\partial\Omega$.
\end{lemma}

\begin{proof}
    First of all, since $\Omega$ is of class $C^{2,\alpha}$, we know we can always extend the vector field $\nu$ smoothly to a tubular neighborhood $\mathcal{U}$ of $\partial\Omega$ (see, e.g., \cite{krpa81}). For instance, we can define $\nu = -\eta \nabla d_{\partial\Omega}$, where the function $d_{\partial\Omega}(x) := \mathrm{dist}(x,\partial\Omega)$ for every $x \in \overline{\Omega}$, and $\eta$ is a smooth cut-off function with support contained in the set in $\mathcal{U}$ and such that $\eta \equiv 1$ near $\partial\Omega$. Moreover, $d_{\partial\Omega} \in C^2(\mathcal{U})$ and $|\nabla d_{\partial\Omega}|^2 = 1$ in $\mathcal{U}$. Furthermore, by \cite[Appendix 14.6]{gitr83}, at any point in $\partial\Omega$, we can choose coordinates so that
    \begin{equation*}
        \nabla \nu = - \nabla^2 d_{\partial\Omega} = 
        \left[\begin{matrix}
        \kappa_1  &\cdots &0 &0 \\
        \vdots &\ddots &\vdots  &\vdots \\
        0  &\cdots &\kappa_{N-1} &0 \\
        0  &\cdots &0 &0
        \end{matrix}
        \right],
    \end{equation*}
    which is the matrix canonically associated to the second fundamental form $\mathrm{I\!I}$ of $\partial\Omega$.

    We recall that, by Theorem \ref{thm regularity C2}, $u \in C^{2,\min\{\alpha,\gamma\}}(\overline{\Omega_\varepsilon})$, for $\varepsilon \in (0,1)$, where $\Omega_\varepsilon=\{x\in\Omega: \mathrm{dist}(x,\partial\Omega)<\varepsilon\}$. Moreover, $\nabla u \neq 0$ in $\overline{\Omega_\varepsilon}$. Without loss of generality, assume that $\varepsilon > 0 $ is sufficiently small in order to have $\Omega_\varepsilon \subset \mathcal{U}$.
    We now proceed assuming to extend the normal $\nu$ as described above, whenever it is required in the computation below.

    Let us define the auxiliary function $v$ representing the Robin boundary condition
    \begin{equation*}
    v := |\nabla u|^{p-2} \langle \nabla u, \nu \rangle + \beta |u|^{p-2}u \quad \text{on } \Omega_\varepsilon.
    \end{equation*}
    Since $v = 0$ on $\partial\Omega$, its gradient $\nabla v$ must be proportional to the normal vector $\nu$ on $\partial\Omega$. We now compute the full gradient of $v$ in the tubular neighborhood $\Omega_\varepsilon$:
    \begin{equation}\label{eq: nabla v}
        \begin{aligned}
        \nabla v &= \nabla \big( |\nabla u|^{p-2} \langle \nabla u, \nu \rangle + \beta |u|^{p-2}u \big)  \\
        &= (p-2)|\nabla u|^{p-4} (\nabla^2 u)\nabla u \langle \nabla u, \nu \rangle + |\nabla u|^{p-2} \nabla ( \langle \nabla u, \nu \rangle ) \\
        &\quad + \beta (p-2)|u|^{p-4}u^2 \nabla u + \beta|u|^{p-2}\nabla u \\
        &= (p-2)|\nabla u|^{p-4} (\nabla^2 u)\nabla u \langle \nabla u, \nu \rangle + |\nabla u|^{p-2} \nabla ( \langle \nabla u, \nu \rangle ) 
        + (p-1) \beta |u|^{p-2} \nabla u \quad \text{in } \Omega_\varepsilon,
        \end{aligned}
    \end{equation}
    where we have used the trivial derivation
    \begin{equation*}
        \nabla (|u|^{p-2}) = (p-2) |u|^{p-3} \frac{u}{|u|} \nabla u, \quad \text{and} \quad \nabla (|\nabla u|^{p-2}) = (p-2) |\nabla u|^{p-3} \frac{\nabla u}{|\nabla u|} (\nabla^2 u).
    \end{equation*}
    By product rule, $\nabla (\langle \nabla u, \nu \rangle) = (\nabla^2 u) \nu + (\nabla \nu) \nabla u$. Furthermore, being the extension of $\nu$ unitary, we have $(\nabla \nu)\nu = 0$ in $\mathcal{U}$, which implies $(\nabla \nu)\nabla u = (\nabla \nu)\nabla_\top u$.
    Moreover, let $\{\theta_i\}_{i=1}^{N-1}$ be the tangential frame given by principal direction of $\partial\Omega$, i.e. $\{ \theta_1, \dots, \theta_{N-1} \}$ is an orthogonal basis at every given point of the tangent space of $\partial\Omega$ (eigenvectors of the second fundamental form $\mathrm{I\!I}$ on $\partial\Omega$) such that 
    \[
    \langle (\nabla \nu) \theta_i, \theta_i \rangle = \mathrm{I\!I}(\theta_i,\theta_i) = \kappa_i.
    \]
    Then, for every $i=1,\dots,N-1$ we have that $\langle \nabla v, \theta_i \rangle = 0$ on $\partial\Omega$. Hence, by \eqref{eq: nabla v}, we have 
    \begin{equation*}
        \begin{aligned}
        0= & \,(p-2)|\nabla u|^{p-4}\langle\nabla u,\nu\rangle\langle(\nabla^2u)\nabla u,\theta_i\rangle +|\nabla u|^{p-2}\langle(\nabla^2u)\nu,\theta_i\rangle
        \\&+|\nabla u|^{p-2}\langle(\nabla\nu)\nabla_\top u,\theta_i\rangle +(p-1)\beta |u|^{p-2}\langle\nabla_\top u,\theta_i\rangle,
        \end{aligned}
    \end{equation*}
    which yields
    \begin{equation}\label{eq: nabla u(p-2)langle...rangle}
        \begin{aligned}
        |\nabla u|^{p-2}\langle(\nabla^2u)\nu,\theta_i\rangle= & -(p-2)|\nabla u|^{p-4}\langle\nabla u,\nu\rangle\langle(\nabla^2u)\nabla u,\theta_i\rangle\\
        &-|\nabla u|^{p-2}\langle(\nabla\nu)\nabla_\top u,\theta_i\rangle -(p-1)\beta |u|^{p-2}\langle\nabla_\top u,\theta_i\rangle \quad \text{on } \partial\Omega.
        \end{aligned}
    \end{equation}
    Therefore, by the standard decomposition $\displaystyle \nabla_\top u = \sum_{i=1}^{N-1}\langle\nabla_\top u,\theta_i\rangle\theta_i$ along with $\nabla u = \nabla_\top u + u_\nu \nu$, and by using \eqref{eq: nabla u(p-2)langle...rangle}, we obtain that
    \begin{equation*}
        \begin{aligned}
        |\nabla u|^{p-2}\langle(\nabla^2u)\nabla u,\nu\rangle&=|\nabla u|^{p-2}\langle\nabla u,\nu\rangle\langle(\nabla^2u) \nu,\nu\rangle +|\nabla u|^{p-2}\sum_{i=1}^{N-1}\langle\nabla_\top u,\theta_i\rangle\langle(\nabla^2u)\theta_i,\nu\rangle\\
        &=|\nabla u|^{p-2}\langle\nabla u,\nu\rangle\langle(\nabla^2u)\nu,\nu\rangle
        -\sum_{i=1}^{N-1}\langle\nabla_\top u,\theta_i\rangle(p-2)|\nabla u|^{p-4}\langle\nabla u,\nu\rangle\langle(\nabla^2u)\nabla u,\theta_i\rangle\\
        & \quad -\sum_{i=1}^{N-1}\langle\nabla_\top u,\theta_i\rangle|\nabla u|^{p-2}\langle(\nabla\nu)\nabla_\top u,\theta_i\rangle
        -\sum_{i=1}^{N-1}\langle\nabla_\top u,\theta_i\rangle(p-1)\beta|u|^{p-2}\langle\nabla_\top u,\theta_i\rangle,
        \end{aligned}
    \end{equation*}
    and the statement follows.
\end{proof}

We are now ready to deduce the Fundamental Identity which will be the starting point of our analysis. It generalizes previous identities obtained for the case $p=2$ with Dirichlet boundary conditions (see \cite[Thm. 2.1]{mapo19} and \cite[Thm. 2.4]{mp2020}) and with Robin boundary conditions (cf. \cite[Thm. 3]{gamo26}).  

\begin{teo}[Fundamental Identity] \label{teo:fundamental}
Let $u \in C^{1,\gamma}(\overline{\Omega})$ be the solution to \eqref{eq:problem_p}. Let $P$ and $\mathcal{L}_u$ be defined by \eqref{eq:P_function} and \eqref{eq: L(phi)_function}, respectively. Then, the following identity holds:
\begin{equation}\label{eq: fundamental identity}
    \begin{aligned}
    \frac{1}{2(p-1)}\int_\Omega \mathcal{L}_u(P) \d{x}+\int_{\partial\Omega} |\nabla u|^{2p-4} \mathrm{I\!I}(\nabla_\top u,\nabla_\top u) \d{}\mathcal{H}^{N-1}+ 2(p-1) \beta\int_{\partial\Omega} |\nabla u|^{p-2}|u|^{p-2}|\nabla_\top u|^2 \d{}\mathcal{H}^{N-1} \\= (N-1) \int_{\partial\Omega} |\nabla u|^{p-2} u_\nu \d{}\mathcal{H}^{N-1}
    - (N-1) \int_{\partial\Omega} |\nabla u|^{2p-4} (u_\nu)^2 \mathrm{M}_{\partial\Omega} \d{}\mathcal{H}^{N-1}.
    \end{aligned}
\end{equation}
\end{teo}

\begin{proof}
    The proof proceed by direct computation and an application of Divergence Theorem. We just recall that, by Proposition \ref{prop properties u} and by Theorem \ref{thm regularity C2}, $u \in C^{2,\alpha}(\Omega\setminus Z)$, $u \in C^{2,\min \{ \alpha,\gamma \}}(\overline{\Omega_\varepsilon})$ and $\nabla u \neq 0$ on $\partial\Omega$. Furthermore, the set $Z$ of critical point of $u$ has zero Lebesgue measure. Hence, the integration and the computation below over $\Omega$ are always intended over $\Omega \setminus Z$. 
    
    We first compute the gradient of the $P$-function defined in \eqref{eq:P_function}. We have
    \begin{equation}\label{eq: nabla P}
        \nabla P = \nabla \left( \frac{2(p-1)}{p}|\nabla u|^p - 2u \right) = 2(p-1)|\nabla u|^{p-2}(\nabla^2 u)\nabla u - 2\nabla u \quad \text{in } \overline{\Omega}.
    \end{equation}
    Then, an integration over $\Omega$ of the function $ \mathcal{L}_u(P)$ along with the Divergence Theorem yields
    \begin{equation} \label{eq:integral_L_P}
        \begin{aligned}
        \int_\Omega \mathcal{L}_u(P) \d{x} &= \int_\Omega \div \left( |\nabla u|^{p-2} A_p(\nabla u)\nabla P \right) \d{x} = \int_{\partial\Omega} |\nabla u|^{p-2} \langle A_p(\nabla u)\nabla P, \nu \rangle \d{}\mathcal{H}^{N-1}.
        \end{aligned}
    \end{equation}
    By \eqref{eq:matrix_A} on the definition of $A_p(\nabla u)$ on non-critical points of $u$ (recall that on $\partial\Omega$ there are no critical points of $u$), we can split the argument of the boundary integral into two terms, namely
    \begin{equation*}
        |\nabla u|^{p-2} \langle A_p(\nabla u)\nabla P, \nu \rangle = \underbrace{|\nabla u|^{p-2} \langle \nabla P, \nu \rangle}_{=:T_1} + \underbrace{(p-2) |\nabla u|^{p-2} \frac{\langle \nabla u, \nabla P \rangle}{|\nabla u|^2} u_\nu}_{=:T_2},
    \end{equation*}
    where, by direct computations using \eqref{eq: nabla P},
    \begin{equation*}
        \begin{aligned}
        T_1 &= 2(p-1)|\nabla u|^{2p-4}\langle (\nabla^2 u)\nabla u, \nu \rangle - 2|\nabla u|^{p-2} u_\nu, 
        \\
        T_2 &= 2(p-1)(p-2)|\nabla u|^{2p-6} u_\nu \langle (\nabla^2 u)\nabla u, \nabla u \rangle - 2(p-2) |\nabla u|^{p-2} u_\nu.
        \end{aligned}
    \end{equation*}
    We now substitute the expression for $T_1+T_2$ into the argument of the boundary integral in \eqref{eq:integral_L_P} and we compute each resulting term. In particular, we first evaluate the boundary integral of the term $T_1$. To this aim, we use the formula for $|\nabla u|^{p-2}\langle(\nabla^2u)\nabla u,\nu\rangle$ on $\partial\Omega$ provided by Lemma \ref{lem:tangential}, yielding
    \begin{equation}\label{eq: int T_1}
        \begin{aligned}
            \int_{\partial\Omega} T_1 \d{}\mathcal{H}^{N-1} =& \,2(p-1)\int_{\partial\Omega}|\nabla u|^{2p-4} u_\nu \langle(\nabla^2u)\nu,\nu\rangle\d{}\mathcal{H}^{N-1}\\
            &-2(p-1)(p-2)\int_{\partial\Omega}|\nabla u|^{2p-6} u_\nu \langle(\nabla^2u)\nabla u,\nabla_\top u\rangle\d{}\mathcal{H}^{N-1}\\
            &-2(p-1)\int_{\partial\Omega}|\nabla u|^{2p-4} \mathrm{I\!I}(\nabla_\top u,\nabla_\top u) \d{}\mathcal{H}^{N-1}\\
            &-2\beta(p-1)^2\int_{\partial\Omega}|\nabla u|^{p-2}|u|^{p-2}|\nabla_\top u|^{2}\d{}\mathcal{H}^{N-1}\\
            &-2\int_{\partial\Omega}|\nabla u|^{p-2} u_\nu \d{}\mathcal{H}^{N-1}.
        \end{aligned}
    \end{equation}
    On the other hand, for the boundary integral of the term $T_2$, we simply get
    \begin{equation}\label{eq: int T_2}
        \begin{aligned}
        \int_{\partial\Omega} T_2 \d{}\mathcal{H}^{N-1} =
        & \, 2(p-1)(p-2)\int_{\partial\Omega}|\nabla u|^{2p-6} u_\nu \langle(\nabla^2u)\nabla u,\nabla u\rangle\d{}\mathcal{H}^{N-1}\\
        &-2(p-2)\int_{\partial\Omega}|\nabla u|^{p-2} u_\nu \d{}\mathcal{H}^{N-1}.
        \end{aligned}
    \end{equation}
    We now treat the first integral on the right-hand side of \eqref{eq: int T_1}. By Reilly's identity \eqref{eq: reilly formula} we get 
    \begin{equation*}
        \begin{aligned}
        |\nabla u|^{p-2} \langle(\nabla^2u)\nu,\nu\rangle =& \,N - (p-2)|\nabla u|^{p-4}\langle(\nabla^2u)\nabla u,\nabla u\rangle 
        \\
        &- |\nabla u|^{p-2} \Delta_\top u - (N-1) |\nabla u|^{p-2} u_\nu \mathrm{M}_{\partial\Omega} \quad \text{on } \partial\Omega.
        \end{aligned}
    \end{equation*}
    Hence, multiplying the previous relation by $2(p-1)|\nabla u|^{p-2} u_\nu $ and integrating over $\partial\Omega$, we obtain 
    \begin{equation}\label{eq: int u nu nu}
        \begin{aligned}
        2(p-1)\int_{\partial\Omega}|\nabla u|^{2p-4} u_\nu \langle(\nabla^2u)\nu,\nu\rangle\d{}\mathcal{H}^{N-1} 
        =& \, 2(p-1)N\int_{\partial\Omega}|\nabla u|^{p-2} u_\nu\d{}\mathcal{H}^{N-1}
        \\
        &-2(p-1)(p-2)\int_{\partial\Omega}|\nabla u|^{2p-6} u_\nu \langle(\nabla^2u)\nabla u,\nabla u\rangle\d{}\mathcal{H}^{N-1}
        \\
        & -2(p-1)\int_{\partial\Omega}|\nabla u|^{2p-4} u_\nu \Delta_\top u\d{}\mathcal{H}^{N-1} 
        \\
        &-2(p-1)(N-1)\int_{\partial\Omega}|\nabla u|^{2p-4}(u_\nu)^2 \mathrm{M}_{\partial\Omega}\d{}\mathcal{H}^{N-1}.
        \end{aligned}
    \end{equation}
    Now, we note that the term containing the tangential Laplacian can be rewritten using the Robin boundary conditions and the Divergence Theorem \ref{thm: divergence part Om}. In fact, it holds
    \begin{equation*}
        \begin{split}    
        \int_{\partial\Omega}|\nabla u|^{2p-4} u_\nu \Delta_\top u\d{}\mathcal{H}^{N-1}
        &= -\beta \int_{\partial\Omega}|\nabla u|^{p-2} |u|^{p-2} u \, \Delta_\top u\d{}\mathcal{H}^{N-1}
        \\
        &= \beta \int_{\partial\Omega} \langle \nabla_\top \Big(|\nabla u|^{p-2} |u|^{p-2} u \Big), \nabla_\top u \rangle \d{}\mathcal{H}^{N-1}.
        \end{split}
    \end{equation*}
    Now using the fact that
    \begin{equation*}
    \begin{split}
        \nabla (|\nabla u|^{p-2}) = (p-2) |\nabla u|^{p-4} (\nabla^2 u) \nabla u \quad \text{and} \quad
        \nabla (|u|^{p-2} u) = |u|^{p-2} \nabla u + (p-2) |u|^{p-4} u^2 \nabla u,
    \end{split}
    \end{equation*}
    we deduce that
    \begin{equation*}
    \begin{split}
        \nabla \Big(|\nabla u|^{p-2} |u|^{p-2} u \Big) &= |\nabla u|^{p-2} |u|^{p-2} \nabla u + (p-2) |\nabla u|^{p-2} |u|^{p-4} u^2 \nabla u 
        \\
        &\quad+ (p-2) |\nabla u|^{p-4} |u|^{p-2} u (\nabla^2 u) \nabla u.
    \end{split}
    \end{equation*}
    Hence, simplifying and using the definition of tangential gradient, i.e., $((\nabla^2 u) \nabla u)_\top = (\nabla^2 u) \nabla u - \langle(\nabla^2 u) \nabla u,\nu\rangle \nu$, we get 
    \begin{equation*}
    \begin{split}
        \nabla_\top \Big(|\nabla u|^{p-2} |u|^{p-2} u \Big) &= (p-1) |\nabla u|^{p-2} |u|^{p-2} \nabla_\top u 
         + (p-2) |\nabla u|^{p-4} |u|^{p-2} u (\nabla^2 u) \nabla u 
         \\
        &\quad
         - (p-2) |\nabla u|^{p-4} |u|^{p-2} u \langle(\nabla^2 u) \nabla u,\nu\rangle \nu,
    \end{split}
    \end{equation*}
    Thus, summing up previous relations, we conclude that 
    \begin{equation}\label{eq: int Delta_top u}
        \begin{aligned}
        -2(p-1) \int_{\partial\Omega}|\nabla u|^{2p-4} u_\nu \Delta_\top u \d{}\mathcal{H}^{N-1} =
        &-2(p-1)^2\beta\int_{\partial\Omega}|\nabla u|^{p-2}|u|^{p-2}|\nabla_\top u|^2\d{}\mathcal{H}^{N-1}\\
        &+2(p-1)(p-2)\int_{\partial\Omega}|\nabla u|^{2p-6} u_\nu \langle(\nabla^2u)\nabla u,\nabla_\top u\rangle\d{}\mathcal{H}^{N-1},
        \end{aligned}
    \end{equation}
    where we have used the Robin boundary condition and the orthogonality of $\nu$ with respect to $\nabla_\top u$.
    Finally, substituting \eqref{eq: int Delta_top u} in \eqref{eq: int u nu nu} and the resulting relation in \eqref{eq: int T_1}, along with adding \eqref{eq: int T_2}, we get
    \begin{equation*}
        \begin{aligned}
        &\int_{\partial\Omega} T_1 \d{}\mathcal{H}^{N-1} + \int_{\partial\Omega} T_2 \d{}\mathcal{H}^{N-1}
        = \,2(p-1)N\int_{\partial\Omega}|\nabla u|^{p-2} u_\nu \d{}\mathcal{H}^{N-1}
        \\
        &\underbrace{-2(p-1)(p-2)\int_{\partial\Omega}|\nabla u|^{2p-6} u_\nu \langle(\nabla^2u)\nabla u,\nabla u\rangle\d{}\mathcal{H}^{N-1}}
        -2(p-1)^2\beta\int_{\partial\Omega}|\nabla u|^{p-2}|u|^{p-2}|\nabla_\top u|^2\d{}\mathcal{H}^{N-1}
        \\
        &\underbrace{+2(p-1)(p-2)\int_{\partial\Omega}|\nabla u|^{2p-6} u_\nu \langle(\nabla^2u)\nabla u,\nabla_\top u\rangle\d{}\mathcal{H}^{N-1}}
        -2(p-1)(N-1)\int_{\partial\Omega}|\nabla u|^{2p-4} (u_\nu)^2 \mathrm{M}_{\partial\Omega}\d{}\mathcal{H}^{N-1}
        \\
        &\underbrace{-2(p-1)(p-2)\int_{\partial\Omega}|\nabla u|^{2p-6} u_\nu \langle(\nabla^2u)\nabla u,\nabla_\top u\rangle\d{}\mathcal{H}^{N-1}}
        -2(p-1)\int_{\partial\Omega}|\nabla u|^{2p-4} \mathrm{I\!I}(\nabla_\top u,\nabla_\top u) \d{}\mathcal{H}^{N-1}
        \\ 
        & -2(p-1)^2\beta\int_{\partial\Omega}|\nabla u|^{p-2}|u|^{p-2}|\nabla_\top u|^{2}\d{}\mathcal{H}^{N-1}
        -2\int_{\partial\Omega}|\nabla u|^{p-2} u_\nu \d{}\mathcal{H}^{N-1}
        \\
        & \underbrace{+2(p-1)(p-2)\int_{\partial\Omega}|\nabla u|^{2p-6} u_\nu \langle(\nabla^2u)\nabla u,\nabla u\rangle\d{}\mathcal{H}^{N-1}}
        -2(p-2)\int_{\partial\Omega}|\nabla u|^{p-2} u_\nu \d{}\mathcal{H}^{N-1}.
        \end{aligned}
    \end{equation*}
    Simplifying the terms in brackets and rearranging using the trivial facts
    \begin{equation*}
    \begin{split}
        &2(p-1)N - 2 - 2(p-2) = 2(p-1)(N-1),
        \\
        &2(p-1)^2\beta + 2(p-1)^2\beta = 4(p-1)^2\beta,
    \end{split}
    \end{equation*}
    yields
    \begin{equation*}
        \begin{aligned}
        &\int_{\partial\Omega} T_1 \d{}\mathcal{H}^{N-1} + \int_{\partial\Omega} T_2 \d{}\mathcal{H}^{N-1} 
        \\
        &= \,2(p-1)(N-1)\int_{\partial\Omega}|\nabla u|^{p-2} u_\nu \d{}\mathcal{H}^{N-1}
        - 4(p-1)^2\beta\int_{\partial\Omega}|\nabla u|^{p-2}|u|^{p-2}|\nabla_\top u|^2\d{}\mathcal{H}^{N-1}
        \\
        &\quad -2(p-1)(N-1)\int_{\partial\Omega}|\nabla u|^{2p-4}(u_\nu)^2 \mathrm{M}_{\partial\Omega}\d{}\mathcal{H}^{N-1}
        -2(p-1)\int_{\partial\Omega}|\nabla u|^{2p-4} \mathrm{I\!I}(\nabla_\top u,\nabla_\top u) \d{}\mathcal{H}^{N-1} 
        .
        \end{aligned}
    \end{equation*}
    Then, the identity of the statement follows by simply dividing by $2(p-1)$.
\end{proof}

We now adapt previous identity in order to obtain a formula suitable to prove the celebrated Soap Bubble Theorem using a solution of \eqref{eq:problem_p}. The following identity generalize to the $p$-Laplacian setting previously obtained relations: we refer to \cite[Thm. 4]{gamo26} and \cite[Thm. 2.2]{mapo19} for the Euclidean space $\R^N$ and \cite{RuHuCh24} for smooth Riemannian manifolds.

\begin{teo}[Identity for Alexandrov's Theorem] \label{teo: identity soap}
Let $u \in C^{1,\gamma}(\overline{\Omega})$ be the solution to \eqref{eq:problem_p}. Let $P$ and $\mathcal{L}_u$ be defined by \eqref{eq:P_function} and \eqref{eq: L(phi)_function}, respectively. Then, the following identity holds:
\begin{equation*}
    \begin{aligned}
    \frac{1}{2(p-1)}\int_\Omega \mathcal{L}_u(P) \d{x} 
    +\int_{\partial\Omega} |\nabla u|^{2p-4} \mathrm{I\!I}(\nabla_\top u,\nabla_\top u) \d{}\mathcal{H}^{N-1}
    + 2(p-1)\beta \int_{\partial\Omega} |\nabla u|^{p-2}|u|^{p-2}|\nabla_\top u|^2 \d{}\mathcal{H}^{N-1} 
    \\
    + (N-1) \int_{\partial\Omega} \Big(|\nabla u|^{p-2} u_\nu - R \Big)^2  \mathrm{M}_{0} \d{}\mathcal{H}^{N-1} 
    =
    - (N-1) \int_{\partial\Omega} |\nabla u|^{2p-4} (u_\nu)^2 (\mathrm{M}_{\partial\Omega} - \mathrm{M}_{0} ) \d{}\mathcal{H}^{N-1},  
    \end{aligned}
\end{equation*}
where
\begin{equation*}
    R = \frac{N|\Omega|}{|\partial\Omega|} \quad\text{and} \quad \mathrm{M}_0 = \frac{1}{R}. 
\end{equation*}
\end{teo}

\begin{proof}
    We work on the identity provided by Theorem \ref{teo:fundamental}. We notice that, adding and subtracting the term
    \begin{equation*}
        (N-1) \int_{\partial\Omega} |\nabla u|^{2p-4} (u_\nu)^2 \mathrm{M}_{0} \d{}\mathcal{H}^{N-1},
    \end{equation*}
    we deduce that
    \begin{equation*}
    \begin{aligned}
        &(N-1) \int_{\partial\Omega} |\nabla u|^{p-2} u_\nu \d{}\mathcal{H}^{N-1} - (N-1) \int_{\partial\Omega} |\nabla u|^{2p-4} (u_\nu)^2 \mathrm{M}_{\partial\Omega} \d{}\mathcal{H}^{N-1}
        \\
        &= - (N-1) \int_{\partial\Omega} |\nabla u|^{2p-4} (u_\nu)^2 (\mathrm{M}_{\partial\Omega} - \mathrm{M}_{0} ) \d{}\mathcal{H}^{N-1}  
        - (N-1) \int_{\partial\Omega} |\nabla u|^{2p-4} (u_\nu)^2 \mathrm{M}_{0} \d{}\mathcal{H}^{N-1} 
        \\
        &\quad + (N-1) \int_{\partial\Omega} |\nabla u|^{p-2} u_\nu \d{}\mathcal{H}^{N-1}
    \end{aligned}
    \end{equation*}
    Then, exploiting the trivial relation
    \begin{equation*}
        |\nabla u|^{2p-4} (u_\nu)^2 = \Big(|\nabla u|^{p-2} u_\nu - R \Big)^2 + 2 R|\nabla u|^{p-2} u_\nu -R^2,
    \end{equation*}
    we obtain that
    \begin{equation*}
    \begin{aligned}
        (N-1) \int_{\partial\Omega} |\nabla u|^{2p-4} (u_\nu)^2 \mathrm{M}_{0} \d{}\mathcal{H}^{N-1} 
        &=
        (N-1) \int_{\partial\Omega} \Big(|\nabla u|^{p-2} u_\nu - R \Big)^2  \mathrm{M}_{0} \d{}\mathcal{H}^{N-1} 
        \\
        &\quad+ 2 (N-1) \int_{\partial\Omega} R|\nabla u|^{p-2} u_\nu \mathrm{M}_{0} \d{}\mathcal{H}^{N-1} 
         - (N-1) \int_{\partial\Omega} R^2 \mathrm{M}_{0} \d{}\mathcal{H}^{N-1} 
        \\
        &
        = (N-1) \int_{\partial\Omega} \Big(|\nabla u|^{p-2} u_\nu - R \Big)^2  \mathrm{M}_{0} \d{}\mathcal{H}^{N-1} 
        \\
        &\quad + (N-1) \int_{\partial\Omega} |\nabla u|^{p-2} u_\nu \d{}\mathcal{H}^{N-1} 
    \end{aligned}
    \end{equation*}
    where, in the last equality, we have used that $R^2 \mathrm{M}_{0} = R$ (by the definition of $\mathrm{M}_0$) and the fact that, by Divergence Theorem, one has
    \begin{equation}\label{eq: boundary int of u_nu = N Omega}
        \int_{\partial\Omega} |\nabla u|^{p-2} u_\nu \d{}\mathcal{H}^{N-1} = \int_{\Omega} \Delta_p u \d{}x = N|\Omega| = R |\partial\Omega|. 
    \end{equation}
    Hence, collecting all previous formulas, we deduce that
    \begin{equation*}
    \begin{aligned}
        &(N-1) \int_{\partial\Omega} |\nabla u|^{p-2} u_\nu \d{}\mathcal{H}^{N-1} - (N-1) \int_{\partial\Omega} |\nabla u|^{2p-4} (u_\nu)^2 \mathrm{M}_{\partial\Omega} \d{}\mathcal{H}^{N-1}
        \\
        &= - (N-1) \int_{\partial\Omega} |\nabla u|^{2p-4} (u_\nu)^2 (\mathrm{M}_{\partial\Omega} - \mathrm{M}_{0} ) \d{}\mathcal{H}^{N-1}   - (N-1) \int_{\partial\Omega} \Big(|\nabla u|^{p-2} u_\nu - R \Big)^2  \mathrm{M}_{0} \d{}\mathcal{H}^{N-1},
    \end{aligned}
    \end{equation*}
    and the identity of the statement follows.
\end{proof}
    
We are now in the position to state and prove an identity to study the Serrin's problem for $p$-Laplacian with Robin boundary condition. The idea is to rewrite the Fundamental Identity of Theorem \ref{teo:fundamental} making suitable modifications to obtain the boundary integral of the overdetermined condition plus suitable constants (see \cite[Thm. 5]{gamo26} for the case $p=2$ and $C=C_0$). The sign of the appearing extra-terms will be analyzed in Section \ref{sec: rigidity} to obtain the aimed rigidity results. Let us underline that the following identity is very general and it holds for any sufficiently smooth solution $u$ of problem \eqref{intro eq:problem_p}, without any overdetermined condition in force.

\begin{teo}[Identity for Serrin's type Theorem] \label{teo: identity serrin}
Let $u \in C^{1,\gamma}(\overline{\Omega})$ be the solution to \eqref{eq:problem_p}. Let $P$ and $\mathcal{L}_u$ be defined by \eqref{eq:P_function} and \eqref{eq: L(phi)_function}, respectively. Then, the following identity holds:
\begin{equation}\label{eq: fundamental identity for serrin}
    \begin{aligned}
    &\frac{1}{2(p-1)}\int_\Omega \mathcal{L}_u(P) \d{x} +\int_{\partial\Omega} |\nabla u|^{2p-4} \mathrm{I\!I}(\nabla_\top u,\nabla_\top u) \d{}\mathcal{H}^{N-1} 
    \\
    &+ \frac{3p-4}{2} \beta \int_{\partial\Omega} |\nabla u|^{p-2}|u|^{p-2}|\nabla_\top u|^2 \d{}\mathcal{H}^{N-1} 
    + (p-1) \beta^{\frac{1}{p-1}} \int_{\partial\Omega} \Big(|\nabla u|^{p-2} u_\nu - R \Big)^2  \d{}\mathcal{H}^{N-1}  
    \\
    &
    + (p-1) \beta^{\frac{1}{p-1}}  \int_{\partial\Omega}  \left( \frac{1}{2} \frac{p-2}{p-1} \left(\frac{u_\nu}{|\nabla u|}\right)^{-\frac{p}{p-1}} \left(1- \left(\frac{u_\nu}{|\nabla u|}\right)^2 \right) - \left(1- \left(\frac{u_\nu}{|\nabla u|}\right)^{\frac{p-2}{p-1}} \right) \right) (|\nabla u|^{p-2} u_\nu)^2 \d{}\mathcal{H}^{N-1} 
    \\
    &
    + \left( (p-1) \beta^{\frac{1}{p-1}} + (p-1)N + 1 \right) R^{\frac{1}{p-1}}  \int_{\partial\Omega} \Big( R^{\frac{p-2}{p-1}} - (|\nabla u|^{p-2} u_\nu)^{\frac{p-2}{p-1}} \Big)  \d{}\mathcal{H}^{N-1}
    \\
    &
    = - \beta \int_{\partial\Omega} \Big( |\nabla u|^p + pNu - p |\nabla u|^{p-2} (u_\nu)^2 + \beta(N-1) |u|^p \mathrm{M}_{\partial\Omega}  - C_0 \Big) |u|^{p-2} \d{}\mathcal{H}^{N-1},  
    \end{aligned}
\end{equation}
where
\begin{equation*}
    R = \frac{N|\Omega|}{|\partial\Omega|} \quad \text{and} \quad  C_0 = (1-p)R^{\frac{p}{p-1}} - \left( \frac{R}{\beta} \right)^{\frac{1}{p-1}} \big( (p-1)N + 1 \big). 
\end{equation*}
\end{teo}

\begin{proof}
    We work again on the identity of Theorem \ref{teo:fundamental}. The aim is to make the right-hand side of \eqref{eq: fundamental identity for serrin} appear. Hence, we define
    \begin{equation*}
        f(u):=  |\nabla u|^p + pNu - p |\nabla u|^{p-2} (u_\nu)^2 + \beta(N-1) |u|^p \mathrm{M}_{\partial\Omega} \quad \text{on }\partial\Omega,
    \end{equation*}
    and we multiply this quantity for $\beta |u|^{p-2}$ and we integrate over $\partial\Omega$. We get
    \begin{equation}\label{eq: int beta delta u u^p-2}
    \begin{split}
        \beta \int_{\partial\Omega} f(u) |u|^{p-2} \d{}\mathcal{H}^{N-1} =  \,&\beta \int_{\partial\Omega} |\nabla u|^p |u|^{p-2} \d{}\mathcal{H}^{N-1} + p \beta N  \int_{\partial\Omega}|u|^{p-2} u \d{}\mathcal{H}^{N-1}
        \\
        &-p \beta \int_{\partial\Omega}  |\nabla u|^{p-2} |u|^{p-2} (u_\nu)^2 \d{}\mathcal{H}^{N-1} + (N-1) \int_{\partial\Omega}  |\nabla u|^{2p-4} (u_\nu)^2 \mathrm{M}_{\partial\Omega} \d{}\mathcal{H}^{N-1},
    \end{split}
    \end{equation}
    where in the last term we have used that, by Robin boundary conditions,
    \[
    \beta |u|^{p} \beta |u|^{p-2} = \beta^2 |u|^{p-2}|u|^{p-2} u^2 = |\nabla u|^{2p-4} (u_\nu)^2.
    \]

    Now, we start from the identity \eqref{eq: fundamental identity} provided by Theorem \ref{teo:fundamental}, and we add and subtract the term 
    \begin{equation*}
    \beta \int_{\partial\Omega} |\nabla u|^p |u|^{p-2} \d{}\mathcal{H}^{N-1} +  p \beta N  \int_{\partial\Omega}|u|^{p-2} u \d{}\mathcal{H}^{N-1}
    -p \beta \int_{\partial\Omega}  |\nabla u|^{p-2} |u|^{p-2} (u_\nu)^2 \d{}\mathcal{H}^{N-1},
    \end{equation*}
    in order to obtain the boundary integral $\displaystyle -\beta \int_{\partial\Omega} f(u) |u|^{p-2} \d{}\mathcal{H}^{N-1}$. In fact, by \eqref{eq: fundamental identity} and by \eqref{eq: int beta delta u u^p-2}, we obtain 
    \begin{equation*}
    \begin{aligned}
    &\frac{1}{2(p-1)}\int_\Omega \mathcal{L}_u(P) \d{x}+\int_{\partial\Omega} |\nabla u|^{2p-4} \mathrm{I\!I}(\nabla_\top u,\nabla_\top u) \d{}\mathcal{H}^{N-1}+ 2\beta(p-1) \int_{\partial\Omega} |\nabla u|^{p-2}|u|^{p-2}|\nabla_\top u|^2 \d{}\mathcal{H}^{N-1} \\
    & = (N-1) \int_{\partial\Omega} |\nabla u|^{p-2} u_\nu \d{}\mathcal{H}^{N-1}
    + \beta \int_{\partial\Omega} |\nabla u|^p |u|^{p-2} \d{}\mathcal{H}^{N-1} + p \beta N  \int_{\partial\Omega}|u|^{p-2} u \d{}\mathcal{H}^{N-1}
    \\
    &\quad 
    -p \beta \int_{\partial\Omega}  |\nabla u|^{p-2} |u|^{p-2} (u_\nu)^2 \d{}\mathcal{H}^{N-1} - \beta \int_{\partial\Omega} f(u) |u|^{p-2} \d{}\mathcal{H}^{N-1},
    \end{aligned}
    \end{equation*}
    which, by splitting the term $|\nabla u|^p |u|^{p-2}$ into $|\nabla u|^{p-2} |u|^{p-2} \left(|\nabla_\top u|^2 + (u_\nu)^2\right)$, by using Robin boundary conditions, and by rearranging terms, yields
    \begin{equation*}
    \begin{aligned}
    \frac{1}{2(p-1)}\int_\Omega \mathcal{L}_u(P) \d{x} \, +& \int_{\partial\Omega} |\nabla u|^{2p-4} \mathrm{I\!I}(\nabla_\top u,\nabla_\top u) \d{}\mathcal{H}^{N-1} 
    + (2p-3) \beta \int_{\partial\Omega} |\nabla u|^{p-2}|u|^{p-2}|\nabla_\top u|^2 \d{}\mathcal{H}^{N-1} 
    \\
    =&  -((p-1)N +1) \int_{\partial\Omega} |\nabla u|^{p-2} u_\nu \d{}\mathcal{H}^{N-1}
    - (p-1) \beta \int_{\partial\Omega} |\nabla u|^{p-2} |u|^{p-2} (u_\nu)^2 \d{}\mathcal{H}^{N-1} 
    \\
    &- \beta \int_{\partial\Omega} f(u) |u|^{p-2} \d{}\mathcal{H}^{N-1}.
    \end{aligned}
    \end{equation*}
    
    The conclusion of the statement follows, by adding and subtracting the term $\displaystyle \beta \int_{\partial\Omega} C_0 |u|^{p-2} \d{}\mathcal{H}^{N-1}$ to the right-hand side of previous equation, once we have verified that 
    \begin{equation}\label{eq:C0 u^(p-2) to conclude the proof}
    \begin{split}
    &-((p-1)N +1) \int_{\partial\Omega} |\nabla u|^{p-2} u_\nu \d{}\mathcal{H}^{N-1}
    - (p-1) \beta \int_{\partial\Omega} |\nabla u|^{p-2} |u|^{p-2} (u_\nu)^2 \d{}\mathcal{H}^{N-1} 
    - \beta \int_{\partial\Omega} C_0 |u|^{p-2} \d{}\mathcal{H}^{N-1}
    \\
    &= 
    - (p-1) \beta^{\frac{1}{p-1}} \int_{\partial\Omega} \Big(|\nabla u|^{p-2} u_\nu - R \Big)^2  \d{}\mathcal{H}^{N-1} 
    - (p-1) \beta^{\frac{1}{p-1}} R^{\frac{p}{p-1}}  \int_{\partial\Omega} \Big( R^{\frac{p-2}{p-1}} - (|\nabla u|^{p-2} u_\nu)^{\frac{p-2}{p-1}} \Big)  \d{}\mathcal{H}^{N-1} 
    \\
    &
    - (p-1) \beta^{\frac{1}{p-1}}  \int_{\partial\Omega}  \left( \frac{1}{2} \frac{p-2}{p-1} \left(\frac{u_\nu}{|\nabla u|}\right)^{-\frac{p}{p-1}} \left(1- \left(\frac{u_\nu}{|\nabla u|}\right)^2 \right) - \left(1- \left(\frac{u_\nu}{|\nabla u|}\right)^{\frac{p-2}{p-1}} \right) \right) (|\nabla u|^{p-2} u_\nu)^2 \d{}\mathcal{H}^{N-1} 
    \\
    &
    - \big( (p-1)N + 1 \big) R^{\frac{1}{p-1}}  \int_{\partial\Omega} \Big( R^{\frac{p-2}{p-1}} - (|\nabla u|^{p-2} u_\nu)^{\frac{p-2}{p-1}} \Big)  \d{}\mathcal{H}^{N-1}
    + \frac{p-2}{2}\beta \int_{\partial\Omega} |\nabla u|^{p-2}|u|^{p-2}|\nabla_\top u|^2 \d{}\mathcal{H}^{N-1}.
    \end{split}
    \end{equation}
    Thus, we now aim to establish the validity of \eqref{eq:C0 u^(p-2) to conclude the proof}, proceeding by direct computation.
    By the definition of $C_0$, we have
    \begin{equation*}
        -\beta \int_{\partial\Omega} C_0 |u|^{p-2} \d{}\mathcal{H}^{N-1} = (p-1) \beta \int_{\partial\Omega} R^{\frac{p}{p-1}} |u|^{p-2} \d{}\mathcal{H}^{N-1} + \big( (p-1)N + 1 \big) R^{\frac{1}{p-1}} \int_{\partial\Omega} \beta^{\frac{p-2}{p-1}} |u|^{p-2} \d{}\mathcal{H}^{N-1}.
    \end{equation*}
    By the Robin boundary condition and using the fact that $u < 0$ on $\overline{\Omega}$, we deduce that 
    \begin{equation}\label{eq: u^p-2}
        |u|^{p-2} = \left(\frac{|\nabla u|^{p-2} u_\nu}{\beta}\right)^\frac{p-2}{p-1} \quad \text{on }\partial\Omega,
    \end{equation}
    which implies 
    \begin{align*}
        -\beta \int_{\partial\Omega} C_0 |u|^{p-2} \d{}\mathcal{H}^{N-1} =& \,(p-1) \beta^{\frac{1}{p-1}} \int_{\partial\Omega} R^{\frac{p}{p-1}} (|\nabla u|^{p-2} u_\nu)^{\frac{p-2}{p-1}} \d{}\mathcal{H}^{N-1} 
        \\
        &+ \big( (p-1)N + 1 \big) R^{\frac{1}{p-1}} \int_{\partial\Omega}  (|\nabla u|^{p-2} u_\nu)^{\frac{p-2}{p-1}}  \d{}\mathcal{H}^{N-1}.
    \end{align*}        
    Thus, in view of the above equation, the left-hand side of \eqref{eq:C0 u^(p-2) to conclude the proof} is given by the sum of two contributions, namely
    \begin{equation}\label{eq:C0 u^(p-2) to conclude the proof 2}
    \begin{split}
    & \underbrace{- \big( (p-1)N + 1 \big)  \int_{\partial\Omega} |\nabla u|^{p-2} u_\nu \d{}\mathcal{H}^{N-1}
    + \big( (p-1)N + 1 \big) R^{\frac{1}{p-1}} \int_{\partial\Omega}  (|\nabla u|^{p-2} u_\nu)^{\frac{p-2}{p-1}}  \d{}\mathcal{H}^{N-1}}_{=:Q_1}
    \\
    & 
    \underbrace{- (p-1) \beta \int_{\partial\Omega} |\nabla u|^{p-2} |u|^{p-2} (u_\nu)^2 \d{}\mathcal{H}^{N-1} 
    + (p-1) \beta^{\frac{1}{p-1}} \int_{\partial\Omega} R^{\frac{p}{p-1}} (|\nabla u|^{p-2} u_\nu)^{\frac{p-2}{p-1}} \d{}\mathcal{H}^{N-1}}_{=:Q_2}.
    \end{split}
    \end{equation}
    First we notice that, by simply using \eqref{eq: boundary int of u_nu = N Omega}, one obtains 
    \begin{equation}\label{eq: Q1}
        Q_1 = - \big( (p-1)N + 1 \big) R^{\frac{1}{p-1}}  \int_{\partial\Omega} \Big( R^{\frac{p-2}{p-1}} - (|\nabla u|^{p-2} u_\nu)^{\frac{p-2}{p-1}} \Big)  \d{}\mathcal{H}^{N-1}.
    \end{equation}
    Hence, it remains to treat $Q_2$. By \eqref{eq: u^p-2} and by adding and subtracting the term $\displaystyle (p-1) \beta^{\frac{1}{p-1}}  R^2 |\partial\Omega|$, we deduce that 
    \begin{equation}\label{eq: Q_2 first}
    \begin{split}
    Q_2  =& - (p-1) \beta^{\frac{1}{p-1}} \int_{\partial\Omega} |\nabla u|^{p-2} (|\nabla u|^{p-2} u_\nu)^{\frac{p-2}{p-1}} (u_\nu)^2 \d{}\mathcal{H}^{N-1} 
    \\
    & + (p-1) \beta^{\frac{1}{p-1}} \int_{\partial\Omega} R^{\frac{p}{p-1}} (|\nabla u|^{p-2} u_\nu)^{\frac{p-2}{p-1}} \d{}\mathcal{H}^{N-1}
    \\
    &  + (p-1) \beta^{\frac{1}{p-1}} \int_{\partial\Omega} R^2 \d{}\mathcal{H}^{N-1} - (p-1) \beta^{\frac{1}{p-1}} \int_{\partial\Omega} R^2 \d{}\mathcal{H}^{N-1} 
    \\
     = & - (p-1) \beta^{\frac{1}{p-1}} R^{\frac{p}{p-1}} \int_{\partial\Omega} \Big( R^{\frac{p-2}{p-1}} - (|\nabla u|^{p-2} u_\nu)^{\frac{p-2}{p-1}} \Big) \d{}\mathcal{H}^{N-1} 
    \\
    &  + (p-1) \beta^{\frac{1}{p-1}} \int_{\partial\Omega} \Big( R^2 - |\nabla u|^{p-2} (|\nabla u|^{p-2} u_\nu)^{\frac{p-2}{p-1}} (u_\nu)^2 \Big) \d{}\mathcal{H}^{N-1}. 
    \end{split}
    \end{equation}
    Now note that a direct computation (using the trivial algebraic identity $(p-2)\left( \frac{p-2}{p-1}\right) = p-2 - \frac{p-2}{p-1} $ to compute the exponents of the term $|\nabla u|$), yields
    \begin{equation}\label{eq: int R^2 - () u_nu^2}
    \begin{split}
        \int_{\partial\Omega} \Big( R^2 - |\nabla u|^{p-2} (|\nabla u|^{p-2} u_\nu)^{\frac{p-2}{p-1}} (u_\nu)^2 \Big) \d{}\mathcal{H}^{N-1} =& \int_{\partial\Omega} \left( R^2 - \left(\frac{u_\nu}{|\nabla u|}\right)^{\frac{p-2}{p-1}} (|\nabla u|^{p-2} u_\nu)^2 \right) \d{}\mathcal{H}^{N-1} 
        \\
        = & - \int_{\partial\Omega} \Big(|\nabla u|^{p-2} u_\nu - R \Big)^2  \d{}\mathcal{H}^{N-1} 
        \\
        &+
        \int_{\partial\Omega} \left(1- \left(\frac{u_\nu}{|\nabla u|}\right)^{\frac{p-2}{p-1}} \right) (|\nabla u|^{p-2} u_\nu)^2 \d{}\mathcal{H}^{N-1}, 
    \end{split}
    \end{equation}
    where, in the last equality, we used the fact that, by \eqref{eq: boundary int of u_nu = N Omega},
    \begin{equation*}
         \int_{\partial\Omega} \Big(|\nabla u|^{p-2} u_\nu - R \Big)^2  \d{}\mathcal{H}^{N-1} =  \int_{\partial\Omega} \Big( (|\nabla u|^{p-2} u_\nu)^2 - R^2 \Big)  \d{}\mathcal{H}^{N-1}.
    \end{equation*}
    Finally, substituting \eqref{eq: int R^2 - () u_nu^2} into \eqref{eq: Q_2 first} and adding and subtracting the last term 
    \begin{equation}\label{eq: final added term for serrin identity}
        (p-1) \beta^{\frac{1}{p-1}}  \int_{\partial\Omega}  \frac{1}{2} \frac{p-2}{p-1} \left(\frac{u_\nu}{|\nabla u|}\right)^{-\frac{p}{p-1}} \left(1- \left(\frac{u_\nu}{|\nabla u|}\right)^2 \right) (|\nabla u|^{p-2} u_\nu)^2 \d{}\mathcal{H}^{N-1},
    \end{equation}
    we obtain
    \begin{equation}\label{eq: Q_2 final}
    \begin{split}
        Q_2 =& - (p-1) \beta^{\frac{1}{p-1}} R^{\frac{p}{p-1}} \int_{\partial\Omega} \Big( R^{\frac{p-2}{p-1}} - (|\nabla u|^{p-2} u_\nu)^{\frac{p-2}{p-1}} \Big) \d{}\mathcal{H}^{N-1}
        \\
        & - (p-1) \beta^{\frac{1}{p-1}} \int_{\partial\Omega} \Big(|\nabla u|^{p-2} u_\nu - R \Big)^2  \d{}\mathcal{H}^{N-1} + \frac{p-2}{2}\beta \int_{\partial\Omega} |\nabla u|^{p-2}|u|^{p-2}|\nabla_\top u|^2 \d{}\mathcal{H}^{N-1}
        \\
        & - (p-1) \beta^{\frac{1}{p-1}}  \int_{\partial\Omega}  \left( \frac{1}{2} \frac{p-2}{p-1} \left(\frac{u_\nu}{|\nabla u|}\right)^{-\frac{p}{p-1}} \left(1- \left(\frac{u_\nu}{|\nabla u|}\right)^2 \right) - \left(1- \left(\frac{u_\nu}{|\nabla u|}\right)^{\frac{p-2}{p-1}} \right) \right) (|\nabla u|^{p-2} u_\nu)^2 \d{}\mathcal{H}^{N-1}, 
    \end{split}
    \end{equation}
    where, we have used that, by 
    \begin{equation*}
        |\nabla u|^2 \left(1- \left(\frac{u_\nu}{|\nabla u|}\right)^2 \right) = |\nabla_\top u|^2,\quad  (u_\nu)^{-\frac{p}{p-1}} (u_\nu)^2 = (u_\nu)^{\frac{p-2}{p-1}},\quad |\nabla u|^{\frac{p}{p-1}} |\nabla u|^{-2} |\nabla u|^{2p-4} = |\nabla u|^{p-2} |\nabla u|^{\frac{(p-2)^2}{p-1}}
    \end{equation*}
    and by the relation \eqref{eq: u^p-2}, the term in \eqref{eq: final added term for serrin identity} is equal to
    \begin{equation*}
        \frac{p-2}{2}\beta \int_{\partial\Omega} |\nabla u|^{p-2}|u|^{p-2}|\nabla_\top u|^2 \d{}\mathcal{H}^{N-1}.
    \end{equation*}
    Thus, substituting \eqref{eq: Q1} and \eqref{eq: Q_2 final} into \eqref{eq:C0 u^(p-2) to conclude the proof 2}, we conclude that \eqref{eq:C0 u^(p-2) to conclude the proof} holds and the identity of the statement follows at once.
    \end{proof}

    \section{Rigidity results and overdetermined problems}\label{sec: rigidity}
    This section is devoted to the analysis of the overdetermined problem \eqref{intro eq: overdetermined problem} described in the Introduction and to prove the main result of the paper, namely Theorem \ref{intro main thm rigidity serrin}. First we state three useful lemmas from which we will deduce the sign of some integral appearing in the identities studied in Section \ref{sec: integral identities}.
    
    We start with a lemma on the positivity of the linearized operator $\mathcal{L}_u(P)$ in $\Omega \setminus Z$ where $Z$ is the set of critical points of $u$. The lemma holds in general for $(M, g)$ a $n$-dimensional complete noncompact Riemannian manifold with nonnegative Ricci curvature, and $\Omega \subset M$ a smooth bounded connected domain (see \cite[Lem.\,4]{RuHuCh24}).

    \begin{lemma}\label{lemma positivity L(P)}
        Let $\Omega$ be an open bounded connected subset of $\R^N$ with $C^{2,\alpha}$ boundary, for $\alpha \in (0,1)$. Let $u \in C^{1,\gamma}(\overline{\Omega})$ be the solution to \eqref{eq:problem_p}. Let $P$ and $\mathcal{L}_u$ be defined by \eqref{eq:P_function} and \eqref{eq: L(phi)_function}, respectively. Then
        \begin{equation*}
            \mathcal{L}_u(P) \geq 0
        \end{equation*}
        in $\Omega\setminus Z$, where $Z := \{x \in \Omega \colon |\nabla u| = 0\}$. Moreover, 
        \begin{equation*}
            \mathcal{L}_u(P) = 0  \iff u \text{ is radially symmetric }.
        \end{equation*}
    \end{lemma}

    \begin{proof}
    We aim to prove the inequality of the statement away from critical points of $u$. 
    Now let $u$ be a solution to \eqref{intro eq:problem_p}. Define the vector field $a_p \colon \R^N \to \R^N, \xi \mapsto a_p(\xi):= |\xi|^{p-2} \xi$ and notice that 
    \begin{equation*}
        \nabla_\xi a_p(\xi) = |\xi|^{p-2} A_p(\xi),
    \end{equation*}
    where $A_p$ is the $R^{N\times N}$ matrix given by \eqref{eq:matrix_A}.
    Therefore, by differentiating the equation $N = \Delta_p u = \div(a_p(\nabla u))$
    for every $k=1,\dots,N$, and simply interchanging partial derivatives and using chain rule, we get
    \begin{equation}\label{eq:linearized_derivative}
    \begin{split}
        0 = \frac{\partial}{\partial x_k}(\Delta_p u) = \div \left( \frac{\partial}{\partial x_k}(a_p(\nabla u)) \right) &= \div \left( \nabla_\xi a_p (\nabla u) \frac{\partial}{\partial x_k} (\nabla u)\right) 
        \\
        &= \div \left( |\nabla u|^{p-2} A_p(\nabla u) \nabla \left( \frac{\partial u}{\partial x_k} \right) \right)
        =\mathcal{L}_u((\nabla u)_{k}).
    \end{split}
    \end{equation}
    (cf. \eqref{eq: L(phi)_function} for the definition of $\mathcal{L}_u$). In particular, we notice that, intending coloum-wise the divergence of a matrix, 
    \eqref{eq:linearized_derivative} implies
    \begin{equation}\label{eq:linearized_derivative2}
    \div \left( |\nabla u|^{p-2} A_p(\nabla u) \nabla^2 u \right) = 0       
    \end{equation}
    as a vector of $\R^N$.
    
    We now compute, for a general $F: \R^N \to \R$, $F\in C^2(\R^N)$, the quantity $\mathcal{L}_u(F(\nabla u))$ on $\Omega \setminus Z$. By chain rule and by Schwartz Theorem, 
    \[
    \nabla (F(\nabla u)) = \nabla^2 u \nabla_\xi F (\nabla u).
    \]
    Moreover, by the product rule for the divergence operator we know that
    \begin{equation*}
        \div(Bv) = \langle \div(B) , v \rangle + \langle B , \nabla v \rangle_{\mathrm{Fr}} \qquad \forall B \in C^1(\R^N, \R^{N\times N}), v \in C^2(\Omega\setminus Z, \R^N),
    \end{equation*}
    where the $\langle \cdot , \cdot \rangle_{\mathrm{Fr}}$ denoted the standard Frobenious scalar product of matrices of $\R^{N\times N}$. Then, we have 
    \begin{equation*}
    \begin{split}
        \mathcal{L}_u(F(\nabla u)) &= 
        \div \left( |\nabla u|^{p-2} A_p(\nabla u)  \nabla (F(\nabla u)) \right)
        =
        \div \left( |\nabla u|^{p-2} A_p(\nabla u) \nabla^2 u \nabla_\xi F (\nabla u) \right)
        =
        \\
        &=\langle \div(|\nabla u|^{p-2} A_p(\nabla u) \nabla^2 u ) , \nabla_\xi F (\nabla u) \rangle + \langle |\nabla u|^{p-2} A_p(\nabla u) \nabla^2 u, \nabla (\nabla_\xi F (\nabla u) ) \rangle_{\mathrm{Fr}}
        \\
        &= \langle |\nabla u|^{p-2} A_p(\nabla u) \nabla^2 u, \nabla (\nabla_\xi F (\nabla u) ) \rangle_{\mathrm{Fr}},
    \end{split}
    \end{equation*}
    where we have used \eqref{eq:linearized_derivative2}. 
    
    Now, in the particular case $F(\xi)= |\xi|^p$, and using the simple relations
    \begin{equation*}
        \nabla_\xi (|\xi|^p)=p|\xi|^{p-2} \xi,
        \quad\text{and} \quad
        \nabla^2_\xi (|\xi|^p) = p|\xi|^{p-2} A_p(\xi),
    \end{equation*}
    we get
    \begin{equation*}
        \mathcal{L}_u(|\nabla u|^{p}) = p \langle |\nabla u|^{p-2} A_p(\nabla u) \nabla^2 u, |\nabla u|^{p-2} A_p(\nabla u) \nabla^2 u \rangle_{\mathrm{Fr}} = p  \left\||\nabla u|^{p-2}A_p(\nabla u)D^2u \right\|^2.
    \end{equation*}
    Finally, from the definition of $P$ and using the fact that $\mathcal{L}_u(u) = (p-1) \Delta_p u = (p-1) \frac{(\Delta_p u)^2}{N}$ (cf. \eqref{eq: L(u)_u}), we deduce that
    \begin{equation*}
        \mathcal{L}_u(P) = 2(p-1) \left\{ \left\||\nabla u|^{p-2}A_p(\nabla u)D^2u \right\|^2 - \frac{(\Delta_p u)^2}{N} \right\}.
    \end{equation*}
    Therefore, $\mathcal{L}_u(P) \geq 0$ in $\Omega \setminus Z$ follows at once by Newton's inequality, see Proposition \ref{prop newton}, being 
    \[
    \Delta_p u = \mathrm{tr}(A) \quad \text{ with } A:= |\nabla u|^{p-2}A_p(\nabla u)D^2u.
    \]
    Moreover, again by Proposition \ref{prop newton}, 
    \[
    \mathcal{L}_u(P) = 0 \iff |\nabla u|^{p-2}A_p(\nabla u)D^2u = k I_N,
    \]
    for a suitable constant $k \in \R$. Since $\Delta_p u = N$, then $k = 1$. Hence
    \begin{equation}\label{eq I_N = grad}
        I_N = |\nabla u|^{p-2}A_p(\nabla u)D^2u = \nabla ( a_p(\nabla u) ) = \nabla ( |\nabla u|^{p-2} \nabla u). 
    \end{equation}
    Let $z \in \Omega$ be a global minimum point of $u$. Such a point exists by the Weierstrass theorem, since $u$ attains its global minimum on the compact set $\overline{\Omega}$. Moreover, $z \notin \partial\Omega$. Indeed, if $z \in \partial\Omega$, then the sign of the normal derivative $u_\nu(z)$, dictated by the fact that $z$ is a minimum point, would contradict the Robin boundary condition (using the sign of $u$ on $\partial\Omega$). In particular, $z \in Z$. Therefore, an application of the Fundamental Theorem of Calculus to \eqref{eq I_N = grad} implies
    \[
    |\nabla u(x)|^{p-2} \nabla u(x) = x-z \quad \text{for } x \in \Omega \setminus Z.
    \]
    Using the fact that $u \in C^{1,\gamma}(\overline{\Omega})$ we deduce that
    \[
    \nabla u(x) = |x-z|^{\frac{1}{p-1}} \frac{x-z}{|x-z|} \quad \text{for } x \in \overline{\Omega},
    \]
    which implies that $u$ is radially symmetric on the open bounded connected subset $\Omega$, i.e. there exists a constant $c \in R$ such that
    \[
    u(x) = \frac{p-1}{p} |x-z|^{\frac{p}{p-1}} + c \quad \text{for } x \in \overline{\Omega},
    \]
    and the proof is complete.
    \end{proof}

    Then we proceed with a straightforward lemma on the positivity of the sum of two boundary integrals involving $\nabla_\top u$, under a natural condition on the minimum of the curvatures $\kappa_{\min}$ (cf.\, \eqref{intro k min}) and on $\beta$.
    \begin{lemma}\label{lemma positivity int II + int alpha(p)}
        Let $\Omega$ be an open bounded connected subset of $\R^N$ with $C^{2,\alpha}$ boundary, for $\alpha \in (0,1)$. Let $u \in C^{1,\gamma}(\overline{\Omega})$ be the solution to \eqref{intro eq:problem_p}. Let $\alpha(p)$ be any fixed positive constant depending on $p$. Assume that
        \begin{equation}\label{eq: kmin+alpha beta}
            \kappa_{\min} + \alpha(p) \beta^{\frac{1}{p-1}} \geq 0.
        \end{equation}
        Then,
        \begin{equation}\label{eq: A(u,nabla u,alpha(p),beta)}
            \int_{\partial\Omega} |\nabla u|^{2p-4} \mathrm{I\!I}(\nabla_\top u,\nabla_\top u) \d{}\mathcal{H}^{N-1} + \alpha(p) \beta \int_{\partial\Omega} |\nabla u|^{p-2}|u|^{p-2}|\nabla_\top u|^2 \d{}\mathcal{H}^{N-1} \geq 0
        \end{equation}
    \end{lemma}

    \begin{proof}
    Let $ \mathcal{A}(u,\nabla u, \alpha(p),\beta)$ be the left-hand side of \eqref{eq: A(u,nabla u,alpha(p),beta)}. Then, it suffices to notice that, by the definition of $\kappa_{\min}$ and by Robin boundary conditions (along with $|\nabla u|\geq u_\nu$ on $\partial\Omega$ and $u < 0$ on $\overline{\Omega}$), one has
    \begin{equation*}
        \kappa_{\min} |\nabla_\top u|^2 \leq \mathrm{I\!I}(\nabla_\top u,\nabla_\top u) \quad \text{and} \quad |\nabla u| \geq \beta^{\frac{1}{p-1}} |u| \quad \text{on } \partial\Omega.
    \end{equation*}
    Hence, we conclude that 
    \begin{equation*}
    \begin{split}
        \mathcal{A}(u,\nabla u, \alpha(p),\beta) &\geq \int_{\partial\Omega} |\nabla u|^{p-2} \beta^{\frac{p-2}{p-1}} |u|^{p-2} \kappa_{\min} |\nabla_\top u|^2  \d{}\mathcal{H}^{N-1} + \alpha(p)\beta \int_{\partial\Omega} |\nabla u|^{p-2}|u|^{p-2}|\nabla_\top u|^2 \d{}\mathcal{H}^{N-1}
        \\
        &\geq \left( \kappa_{\min} + \alpha(p) \beta^{\frac{1}{p-1}} \right) \beta^{\frac{p-2}{p-1}} \int_{\partial\Omega} |\nabla u|^{p-2}|u|^{p-2}|\nabla_\top u|^2
        \d{}\mathcal{H}^{N-1} \geq 0,
    \end{split}
    \end{equation*}
    thanks to the assumption \eqref{eq: kmin+alpha beta}. Thus, \eqref{eq: A(u,nabla u,alpha(p),beta)} holds.
    \end{proof}

    \begin{rem}\label{remark curvatures}
        For a open bounded domain $\Omega$ of class $C^{1,1}$, it is well known that interior and exterior ball conditions are satisfied. In particular, it holds that 
        \begin{equation*}
        \kappa_{\min} |\nabla_\top u|^2 \leq \mathrm{I\!I}(\nabla_\top u,\nabla_\top u) \leq \kappa_{\max} |\nabla_\top u|^2  \quad \text{on } \partial\Omega,
    \end{equation*}
    where
    \[
    \kappa_{\min} = \min_{\partial \Omega} \min_{i=1,\dots,N-1}  \kappa_i(x) \quad \text{and} \quad \kappa_{\max} = \max_{\partial \Omega} \max_{i=1,\dots,N-1}  \kappa_i(x).
    \]
    Moreover one has
    \[
    -\frac{1}{r_e} \leq  \kappa_{\min} \leq \kappa_{\max} \leq \frac{1}{r_i}
    \]
    where $r_i$ and $r_e$ are the radii of the uniform interior and exterior sphere conditions (see \cite[Rmk. 2.2]{mmp}).
    \end{rem}

    Now we provide a lemma which follows by an standard application of Jensen inequality to a suitable power of the solution $u$. 

    \begin{lemma}\label{lem:jensen ineq}
        Let $\Omega$ be an open bounded connected subset of $\R^N$ with $C^{2,\alpha}$ boundary, for $\alpha \in (0,1)$. Let $u \in C^{1,\gamma}(\overline{\Omega})$ be the solution to \eqref{intro eq:problem_p}. Then
        \begin{equation*}
            \int_{\partial\Omega} (|\nabla u|^{p-2} u_\nu)^{\frac{p-2}{p-1}} \d{}\mathcal{H}^{N-1} \leq R^{\frac{p-2}{p-1}} |\partial\Omega|.
        \end{equation*}
    \end{lemma}

    \begin{proof}
    We proceed by applying the Jensen inequality to the function $f= |u|^{p-2}$ with $q = \frac{p-1}{p-2} > 1$. This implies that
    \begin{equation*}
        \left(\frac{1}{|\partial\Omega|} \int_{\partial\Omega} |u|^{p-2} \d{}\mathcal{H}^{N-1}  \right)^{\frac{p-1}{p-2}} \leq \frac{1}{|\partial\Omega|} \int_{\partial\Omega} |u|^{p-1} \d{}\mathcal{H}^{N-1} = - \frac{1}{|\partial\Omega|} \int_{\partial\Omega} |u|^{p-2}u \d{}\mathcal{H}^{N-1} = \frac{R}{\beta},
    \end{equation*}
    where in the last two equality we have used the sign of $u$, the Robin boundary condition and \eqref{eq: boundary int of u_nu = N Omega}. Moreover, raising the above inequality to the power $\frac{p-2}{p-1}$ we deduce that 
    \[
    \int_{\partial\Omega} |u|^{p-2} \d{}\mathcal{H}^{N-1} \leq \left( \frac{R}{\beta}\right)^{\frac{p-2}{p-1}} |\partial\Omega|,
    \]
    and by the condition $\beta^{\frac{p-2}{p-1}} |u|^{p-2} = (|\nabla u|^{p-2} u_\nu)^{\frac{p-2}{p-1}}$ on $\partial\Omega$ (cf. \eqref{eq: u^p-2}), we conclude. 
    \end{proof}

    We now prove a central lemma providing an inequality which will be used in the sequel to prove again the positivity of certain boundary integrals in the identity \eqref{eq: fundamental identity for serrin} of Theorem \ref{teo: identity serrin}. In particular, we underline that, when $p=2$, the lemma is trivially true and the aforementioned boundary integral are null. The following holds.
    
    \begin{lemma}\label{lem:rho ineq}
        Let $\rho \in [0,1]$, $p \geq 2$. Then, it holds that
        \begin{equation}\label{ineq: 1-rho^alpha < 12(1-rho^2)}
            1-\rho^{\frac{p-2}{p-1}} \leq \frac{1}{2} \frac{p-2}{p-1} \rho^{-\frac{p}{p-1}} (1-\rho^2),
        \end{equation}
        with the convention that the right-hand side of \eqref{ineq: 1-rho^alpha < 12(1-rho^2)} reduces to  $+\infty$ for $\rho=0$.
    \end{lemma}

    \begin{proof}
        For $\rho=0$ there is nothing to prove, being the convention in the statement in force, while for $\rho=1$ the inequality trivially holds.
        Hence, it suffices to prove \eqref{ineq: 1-rho^alpha < 12(1-rho^2)} for $\rho \in (0,1)$. Notice that, $p\geq 2$ implies $\frac{p-2}{p-1} \in [0,1)$. Moreover,
        \begin{equation*}
            \frac{p}{p-1} = 2-\frac{p-2}{p-1}.
        \end{equation*}
        Consider the function $g: (0,+\infty) \longrightarrow \R, s \mapsto g(s) := s^{1- \frac{1}{2}\frac{p-2}{p-1}}$. The map $g$ is concave since $1- \frac{1}{2}\frac{p-2}{p-1} \in (\frac{1}{2},1]$. Hence, by the tangent property of concave functions at the point $s=1$, we have
        \begin{equation*}
            s^{1- \frac{1}{2}\frac{p-2}{p-1}} \leq 1 + \left(1- \frac{1}{2}\frac{p-2}{p-1} \right)(s-1) \quad \forall s \in (0,+\infty).
        \end{equation*}
        Simplifying and adding $s$ to both sides, we get
        \begin{equation*}
            s^{\frac{1}{2}\frac{p-2}{p-1}} - s \leq 1 + \frac{1}{2}\frac{p-2}{p-1}(s-1) - s = \frac{1}{2}\frac{p}{p-1}(1-s) \quad \forall s \in (0,+\infty).
        \end{equation*}
        In particular, considering simply $s = \rho^2$ for $\rho \in (0,1)$, we obtain
        \begin{equation*}
            \rho^{2 - \frac{p-2}{p-1}} - \rho^2 \leq \frac{1}{2}\frac{p-2}{p-1}(1-\rho^2) \quad \forall \rho \in (0,1),
        \end{equation*}
        and multiplying by $\rho^{-\frac{p}{p-1}}$ yields \eqref{ineq: 1-rho^alpha < 12(1-rho^2)}.
    \end{proof}

    We prove a first rigidity result in the spirit of \cite[Thm.\,8]{gamo26} which we deduce directly from the Fundamental Identity of Theorem \ref{teo:fundamental}.

    \begin{teo}\label{teo: first rigidity result}
        Let $\Omega$ be an open bounded connected subset of $\R^N$ with $C^{2,\alpha}$ boundary, for $\alpha \in (0,1)$. Let $u \in C^{1,\gamma}(\overline{\Omega})$ be the solution to \eqref{intro eq:problem_p}.
        Assume that
        \begin{itemize}
            \item[(i)]  $\kappa_{\min} + 2(p-1)\beta^{\frac{1}{p-1}} > 0$;
 
            \item[(ii)] $\displaystyle \int_{\partial\Omega} \Big(|\nabla u|^{p-2} u_\nu \mathrm{M}_{\partial\Omega} -1 \Big) |\nabla u|^{p-2} u_\nu \d{}\mathcal{H}^{N-1} \geq 0$. 
        \end{itemize}
        Then $\Omega$ is a ball and $u$ is radially symmetric.
    \end{teo}

    \begin{proof}
        Lemma \ref{lemma positivity L(P)} and Lemma \ref{lemma positivity int II + int alpha(p)} for $\alpha(p)=2(p-1)$ (cf.\, assumption $(i)$) yield
        \begin{equation*}
            \frac{1}{2(p-1)}\int_\Omega \mathcal{L}_u(P) \d{x}+\int_{\partial\Omega} |\nabla u|^{2p-4} \mathrm{I\!I}(\nabla_\top u,\nabla_\top u) \d{}\mathcal{H}^{N-1}+ 2(p-1) \beta\int_{\partial\Omega} |\nabla u|^{p-2}|u|^{p-2}|\nabla_\top u|^2 \d{}\mathcal{H}^{N-1} \geq 0.
        \end{equation*} 
        On the other hand, 
        \begin{equation*}
        \begin{split}
        (N-1) \int_{\partial\Omega} |\nabla u|^{p-2} u_\nu \d{}\mathcal{H}^{N-1}
        &- (N-1) \int_{\partial\Omega} |\nabla u|^{2p-4} (u_\nu)^2 \mathrm{M}_{\partial\Omega} \d{}\mathcal{H}^{N-1}
        \\
        &=
        - (N-1) \int_{\partial\Omega} \Big(|\nabla u|^{p-2} u_\nu \mathrm{M}_{\partial\Omega} -1 \Big) |\nabla u|^{p-2} u_\nu \d{}\mathcal{H}^{N-1}
        \leq 0,
        \end{split}
    \end{equation*}
    by the assumption $(ii)$. Thus, 
    \begin{equation*}
        \mathcal{L}_u(P) = 0 \quad\text{in } \Omega, \quad \nabla_\top u = 0 \text{ and } |\nabla u|^{p-2} u_\nu = R  \quad \text{on } \partial\Omega,
    \end{equation*}
    and the conclusion follows.
    \end{proof}

    \begin{rem}
        We notice that, a trivial corollary of previous theorem can be obtained assuming the point-wise inequality
        \[
        |\nabla u|^{p-2} u_\nu \mathrm{M}_{\partial\Omega} \geq 1 \quad \text{on } \partial\Omega, 
        \]
        instead of assumption $(ii)$. In particular, the assumption given by the equality case in the above relation has been considered in literature for problems under Dirichlet boundary conditions: see, for instance, \cite[Thm.\,2.2]{mapo19}, \cite[Thm.\,1.1]{CoFe20} and \cite[Thm.\,4]{RuHuCh24} for the analogue result for $p$-Laplacian on manifolds.
    \end{rem}

    Before proving our main result, we describe how the identity of Theorem \ref{teo: identity soap} provides an alternative proof for smooth domains of the celebrated Alexandrov's Theorem, using a suitable $p$-Laplacian Robin boundary value problem (see \cite[Thm.\,2.2]{mapo19}, \cite[Thm.\,7]{gamo26}, \cite[Thm.\,3.1]{CoFe20}).
    
    \begin{teo}[Alexandrov’s Theorem]\label{thm alexandrov}
        Let $\Omega$ be an open bounded connected subset of $\R^N$ with $C^{2,\alpha}$ boundary, for $\alpha \in (0,1)$. If the mean curvature $\mathrm{M}_{\partial\Omega}$ satisfies 
        \begin{equation}\label{eq: cond M > M0}
            \mathrm{M}_{\partial\Omega} \geq \mathrm{M}_{0} \quad \text{on } \partial\Omega, 
        \end{equation}
        where $\mathrm{M}_0 = \frac{1}{R}$, then $\Omega$ is a ball of radius $R$. In particular, the same conclusion holds if $\mathrm{M}_{\partial\Omega}$ equals a constant on $\partial \Omega$.
    \end{teo}

    \begin{proof}
    The claim follows from the identity of Theorem \ref{teo: identity soap}. In fact, since $\Omega$ is of class $C^{2,\alpha}$ (hence it has curvatures uniformly bounded from above and from below, see Remark \ref{remark curvatures}), we can fix a parameter $\beta > 0$ such that
    \begin{equation}\label{eq: assumption k + 2(p-1) beta}
    \kappa_{\min} + 2(p-1)\beta^{\frac{1}{p-1}} > 0.
    \end{equation}
    Then, consider the solution $u \in C^{1,\gamma}(\overline{\Omega}) $ of \eqref{eq:problem_p}. By Lemma \ref{lemma positivity L(P)} and Lemma \ref{lemma positivity int II + int alpha(p)} for $\alpha(p)=2(p-1)$ (cf.\, \eqref{eq: assumption k + 2(p-1) beta}), we know that
    \begin{equation*}
    \begin{aligned}
    \frac{1}{2(p-1)}\int_\Omega \mathcal{L}_u(P) \d{x} 
    +\int_{\partial\Omega} |\nabla u|^{2p-4} \mathrm{I\!I}(\nabla_\top u,\nabla_\top u) \d{}\mathcal{H}^{N-1}
    &+ 2(p-1)\beta \int_{\partial\Omega} |\nabla u|^{p-2}|u|^{p-2}|\nabla_\top u|^2 \d{}\mathcal{H}^{N-1} 
    \\
    &+ (N-1) \int_{\partial\Omega} \Big(|\nabla u|^{p-2} u_\nu - R \Big)^2  \mathrm{M}_{0} \d{}\mathcal{H}^{N-1} 
    \geq 0
    \end{aligned}
    \end{equation*}
    while, by assumption \eqref{eq: cond M > M0}, clearly
    \begin{equation*}
        - (N-1) \int_{\partial\Omega} |\nabla u|^{2p-4} (u_\nu)^2 (\mathrm{M}_{\partial\Omega} - \mathrm{M}_{0} ) \d{}\mathcal{H}^{N-1}  \leq 0
    \end{equation*}
    The conclusion follows as in the last part of the proof of Theorem \ref{teo: first rigidity result}. 

    In particular, if $\mathrm{M}_{\partial\Omega}$ is equal to some constant on $\partial\Omega$, then Minkowski's identity 
    \begin{equation*}
        \int_{\partial\Omega} \mathrm{M}_{\partial\Omega} \langle x-z, \nu \rangle \d{}\mathcal{H}^{N-1} = |\partial\Omega| 
    \end{equation*}
    implies $\mathrm{M}_{\partial\Omega} = \mathrm{M}_{0}$ and we conclude.
\end{proof}

Finally we present a Serrin's type Theorem for $p$-Laplacian with Robin boundary condition under a natural integral condition coming from identity \eqref{eq: fundamental identity for serrin}.

\begin{teo}\label{thm rigidity serrin integral}
    Let $\Omega$ be an open bounded connected subset of $\R^N$ with $C^{2,\alpha}$ boundary, for $\alpha \in (0,1)$. Let $u \in C^{1,\gamma}(\overline{\Omega})$ be the solution to \eqref{intro eq:problem_p}. Let 
    \begin{equation*}
        R = \frac{N|\Omega|}{|\partial\Omega|} \quad \text{and} \quad  C_0 = (1-p)R^{\frac{p}{p-1}} - \left( \frac{R}{\beta} \right)^{\frac{1}{p-1}} \big( (p-1)N + 1 \big). 
    \end{equation*}
    Assume that 
    \begin{itemize}
            \item[(i)]  $\kappa_{\min} +  \frac{3p-4}{2} \beta^{\frac{1}{p-1}} > 0$;
    
            \item[(ii)] $\displaystyle \int_{\partial\Omega} \Big( |\nabla u|^p + pNu - p |\nabla u|^{p-2} (u_\nu)^2 + \beta(N-1) |u|^p \mathrm{M}_{\partial\Omega}  - C_0 \Big) |u|^{p-2} \d{}\mathcal{H}^{N-1} \geq 0$.
    \end{itemize}
    Then $\Omega$ is a ball and $u$ is radially symmetric.
\end{teo}

\begin{proof}
    The strategy of the proof is based on the fundamental identity \eqref{eq: fundamental identity for serrin} provided by Theorem \ref{teo: identity serrin}. First we notice that, by Lemma \ref{lemma positivity L(P)} and by Lemma \ref{lemma positivity int II + int alpha(p)} for $\displaystyle \alpha(p)= \frac{3p-4}{2}$ (cf.\, assumption $(i)$), we know that  
    \begin{equation}\label{eq: 1 lhs positive}
    \begin{aligned}
    &\frac{1}{2(p-1)}\int_\Omega \mathcal{L}_u(P) \d{x} +\int_{\partial\Omega} |\nabla u|^{2p-4} \mathrm{I\!I}(\nabla_\top u,\nabla_\top u) \d{}\mathcal{H}^{N-1} 
    \\
    &+ \frac{3p-4}{2} \beta \int_{\partial\Omega} |\nabla u|^{p-2}|u|^{p-2}|\nabla_\top u|^2 \d{}\mathcal{H}^{N-1} 
    + (p-1) \beta^{\frac{1}{p-1}} \int_{\partial\Omega} \Big(|\nabla u|^{p-2} u_\nu - R \Big)^2  \d{}\mathcal{H}^{N-1}  \geq 0,
    \end{aligned}
    \end{equation}
    Moreover, a simple application of Lemma \ref{lem:jensen ineq} yields 
    \begin{equation}\label{eq: 2 lhs positive}
        \left( (p-1) \beta^{\frac{1}{p-1}} + (p-1)N + 1 \right) R^{\frac{1}{p-1}}  \int_{\partial\Omega} \Big( R^{\frac{p-2}{p-1}} - (|\nabla u|^{p-2} u_\nu)^{\frac{p-2}{p-1}} \Big)  \d{}\mathcal{H}^{N-1} \geq 0
    \end{equation}
    Furthermore, Lemma \ref{lem:rho ineq} for $\displaystyle \rho = \frac{u_\nu}{|\nabla u|} \in (0,1]$ implies 
    \begin{equation}\label{eq: 3 lhs positive}
        (p-1) \beta^{\frac{1}{p-1}}  \int_{\partial\Omega}  \left( \frac{1}{2} \frac{p-2}{p-1} \left(\frac{u_\nu}{|\nabla u|}\right)^{-\frac{p}{p-1}} \left(1- \left(\frac{u_\nu}{|\nabla u|}\right)^2 \right) - \left(1- \left(\frac{u_\nu}{|\nabla u|}\right)^{\frac{p-2}{p-1}} \right) \right) (|\nabla u|^{p-2} u_\nu)^2 \d{}\mathcal{H}^{N-1} \geq 0.
    \end{equation}
    Hence, \eqref{eq: 1 lhs positive}, \eqref{eq: 2 lhs positive} and \eqref{eq: 3 lhs positive} yield the non-negativity of the left-hand side of \eqref{eq: fundamental identity for serrin}. Then, by assumption $(ii)$ we deduce that every integral in \eqref{eq: fundamental identity for serrin} is null, hence
    \begin{equation*}
        \mathcal{L}_u(P) = 0 \quad\text{in } \Omega, \quad \nabla_\top u = 0 \text{ and } |\nabla u|^{p-2} u_\nu = R  \quad \text{on } \partial\Omega,
    \end{equation*}
    and the conclusion follows by Lemma \ref{lemma positivity L(P)}.
\end{proof}
    
Assumption $(ii)$ in the statement of Theorem \ref{thm rigidity serrin integral} can be replaced by
\begin{equation*}
    |\nabla u|^p + pNu - p |\nabla u|^{p-2} (u_\nu)^2 + \beta(N-1) |u|^p \mathrm{M}_{\partial\Omega}  \geq C_0 \quad \text{on } \partial\Omega,
\end{equation*}
which is clearly a stronger assumption. In particular, as a corollary we get our main result.

\begin{proof}[Proof of Theorem \ref{intro main thm rigidity serrin}]
    Since 
    \[
    |\nabla u|^p + pNu - p |\nabla u|^{p-2} (u_\nu)^2 + \beta(N-1) |u|^p \mathrm{M}_{\partial\Omega} = C\quad \text{on } \partial\Omega,
    \]
    for a constant $C \geq C_0$, then assumption $(ii)$ of Theorem \ref{thm rigidity serrin integral} holds and the conclusion follows.
\end{proof}

Finally we just state a straightforward corollary of Theorem \ref{thm rigidity serrin integral} valid for convex domains where clearly $\kappa_{\min} > 0$. Along with the condition $\beta>0$, this implies the validity of assumption $(i)$ of Theorem \ref{thm rigidity serrin integral}. Hence, we get the following.

\begin{cor}\label{cor rigidity serrin integral}
     Let $\Omega$ be an open bounded connected subset of $\R^N$ with $C^{2,\alpha}$ boundary, for $\alpha \in (0,1)$. Let $u \in C^{1,\gamma}(\overline{\Omega})$ be the solution to \eqref{intro eq:problem_p}. Let 
    \begin{equation*}
        R = \frac{N|\Omega|}{|\partial\Omega|} \quad \text{and} \quad  C_0 = (1-p)R^{\frac{p}{p-1}} - \left( \frac{R}{\beta} \right)^{\frac{1}{p-1}} \big( (p-1)N + 1 \big). 
    \end{equation*}
    Assume that 
    \[
    \displaystyle \int_{\partial\Omega} \Big( |\nabla u|^p + pNu - p |\nabla u|^{p-2} (u_\nu)^2 + \beta(N-1) |u|^p \mathrm{M}_{\partial\Omega}  - C_0 \Big) |u|^{p-2} \d{}\mathcal{H}^{N-1} \geq 0.
    \]
    Then $\Omega$ is a ball and $u$ is radially symmetric.
\end{cor}

\appendix
\section{Proof of Theorem \ref{thm regularity C2}}\label{sec: appendix reg}

We start by recalling some notation on local charts of smooth bounded subsets of $\R^N$ in order to state a global regularity theorem from \cite{Lieberman12} on solutions of elliptic equations with oblique boundary  conditions.

Let $\Omega$ be an open bounded connected $C^{2,\alpha}$ subset of $\R^N$, for some $\alpha\in(0,1)$ and let $\tilde{x} = (\tilde{x}',\tilde{x}_N) \in\partial\Omega$. After a rigid motion, there exist $r_0>0$ and a function $\omega\in C^{2,\alpha}(B_{r_0}(\tilde{x}'))$, where $B_{r_0}(\tilde{x}')\subset\R^{N-1}$, such that locally
\[
    \Omega = \{x=(x',x_N) \colon x_N>\omega(x')\}.
\]
For $0<r\le r_0$ we define
\begin{equation*}
\begin{aligned}
\Omega[r]&:=\{x=(x',x_N)\colon x_N > \omega(x'),\ |x'|<r\},\\
\Sigma[r]&:=\{x=(x',x_N)\colon x_N = \omega(x'),\ |x'|<r\},\\
\sigma[r]&:=\partial\Omega[r]\setminus\Sigma[r].
\end{aligned}
\end{equation*}
Here $\Sigma[r]$ is the real boundary portion and $\sigma[r]$ is the artificial boundary. On $\Omega[r]$ we define
\begin{equation}\label{eq d and delta}
    d(x):= \mathrm{dist}(x,\partial \Omega[r]),\qquad \delta(x):= \min \{d(x), \mathrm{diam}\,\Omega[r]\}.
\end{equation}
For $b\in\R$, $k\in \{0,1,2\}$, and $0<\theta<1$, we introduce the weighted H\"older seminorms and norms
\[
\begin{aligned}
    |f|_0^{(b)}&=\sup_{x\in\Omega[r]}\delta(x)^b|f(x)|,\\
    [f]_{\theta}^{(b)}&=
    \sup_{\substack{x\ne y\in\Omega[r]\\ |x-y|<d(x)/2}}
    \delta(x)^{b+\theta}\frac{|f(x)-f(y)|}{|x-y|^{\theta}},\\
    |f|_{k,\theta}^{(b)}&=\sum_{|j|=0}^k |\nabla^j f|_0^{(b+|j|)}+[\nabla^k f]_{\theta}^{(b+k)}.
\end{aligned}    
\]
defined for a function $f \in C^k(\overline{\Omega[r]})$, and $j$ a multindex. The corresponding weighted H\"older space is given by
\[
    C_{(b)}^{k,\theta}(\overline{\Omega[r]})=\left\{ f \in C^k(\overline{\Omega[r]}) \colon |f|_{k,\theta}^{(b)}<+\infty\right\}.
\]
We now recall the following global oblique regularity estimate result, see \cite[Proposition 11.21]{Lieberman12}.

\begin{teo}[Global oblique regularity estimate]\label{lieberman-teor}
Let $r \in (0,1)$.  Assume that
\[
    a^{ij}\in C_{(0)}^{0,\alpha}(\overline{\Omega[r]}),\qquad a^0\in C_{(2)}^{0,\alpha}(\overline{\Omega[r]}),
\]
and that $b\in C^{1,\alpha}(\Sigma[r]\times B_K(0))$ for some $K>0$.  Suppose that the matrix $a=(a^{ij})_{i,j=1}^N$ is uniformly elliptic, namely its eigenvalues belong to $[\lambda,\mu\lambda]$ for some $\lambda>0$ and $\mu\ge1$, and that there exist constants $A_1,B_1,\chi>0$ such that
\begin{align}
    |a^{ij}|_{\alpha}^{(0)}+|a^0|_{\alpha}^{(2)}&\le A_1\lambda,
    \label{teor-L1}\\
    \langle \nabla_q b(x,q), \bar{\nu}(x) \rangle &\ge \chi,
    \label{teor-L2}\\
    |\nabla b(x,q)-\nabla b(y,t)|
    &\le B_1\left(|q-t|^{\alpha}+d^\ast(x)^{-1-\alpha}|x-y|^{\alpha}\right)
    \label{teor-L3}
\end{align}
for $x,y\in\Sigma[r]$, $q,t\in B_K(0)$, and $|x-y|<d^\ast(x)/2$, where $d^\ast(x):= \mathrm{dist}(x,\sigma[r])$.  Then, every solution $u \in C^1(\overline{\Omega[r]}) \cap C^2(\Omega[r])$ of
\begin{equation*}
\begin{cases}
    \mathrm{tr} ( a \nabla^2 u) + a^0 =0 & \text{in }\Omega[r],\\
    b(\cdot , \nabla u)=0 & \text{on }\Sigma[r],
\end{cases}
\end{equation*}
is of class $C^{2,\alpha}$ up to the real boundary in the smaller chart $\Omega[r/2]$, i.e.,
\[
    u\in C^{2,\alpha}\left(\overline{\Omega[r/2]}\right).
\]
\end{teo}

We are now in the position to prove Theorem \ref{thm regularity C2}.

\begin{proof}[Proof of Theorem \ref{thm regularity C2}]
Set $\tilde{\alpha}:= \min\{ \alpha, \gamma\}$. By Proposition \ref{prop properties u}, $u \in C^{1,\tilde{\alpha}}(\overline\Omega)$ and $u<0$ in $\overline\Omega$. Hence, the function $v=\log(-u)$ is well defined and belongs to $C^{1,\tilde{\alpha}}(\overline\Omega)$. Moreover,
\[
    u=-e^v,\quad\nabla u=-e^v\nabla v,\quad|\nabla u|^{p-2}\nabla u=-e^{(p-1)v}|\nabla v|^{p-2}\nabla v.
\]
Consequently,
\begin{equation*}
    \Delta_p u =\div\left(-e^{(p-1)v}|\nabla v|^{p-2}\nabla v\right)=-e^{(p-1)v}\left(\Delta_p v+(p-1)|\nabla v|^p\right).
\end{equation*}
Since $\Delta_p u=N$, we obtain
\begin{equation}\label{equation-v}
    \Delta_p v+(p-1)|\nabla v|^p+N e^{-(p-1)v}=0\quad\text{in }\Omega.
\end{equation}
The Robin condition becomes $0=-e^{(p-1)v}\left(|\nabla v|^{p-2}v_\nu+\beta\right)$ on $\partial\Omega$.
Since $\bar{\nu}=-\nu$, this is equivalent to
\begin{equation}\label{boundary-v}
    |\nabla v|^{p-2}\langle\nabla v,\bar{\nu} \rangle-\beta=0\quad\text{on }\partial\Omega.
\end{equation}
By Proposition \ref{prop properties u}, $\nabla u \neq 0$ on $\partial\Omega$. Furthermore, since $u\in C^{1,\tilde{\alpha}}(\overline\Omega)$, there exist $\rho_0>0$ and constants $m,M>0$ such that
\begin{equation}\label{nablav-bounds}
    0<m<|\nabla v|<M\quad\text{in }\overline{\Omega_{\rho_0}}.
\end{equation}
Fix $\tilde{x}\in\partial\Omega$ and choose a boundary local chart $\Omega[r]$, with $r \in (0,1)$ sufficiently small so that $\overline{\Omega[r]} \subset \overline{\Omega_{\rho_0}}$ and $\partial \Omega[r] \cap \partial \Omega_{\rho_0} = \Sigma[r]$.
In $\Omega[r]$, equation \eqref{equation-v} can be written in nondivergence form. For $q\in\R^N$ with $m<|q|<M$, set (cf. \eqref{eq:matrix_A})
\begin{align*}
    \tilde{a}(q)&= |q|^{p-2} A_p(q) = |q|^{p-2} I_N +(p-2)|q|^{p-4} q \otimes q,\\
    \tilde{a}^0(z,q)&=(p-1)|q|^p+N e^{-(p-1)z}.
\end{align*}
Thus \eqref{equation-v} in $\Omega[r]$ is equivalent to
\begin{equation}\label{eq-nondiv}
    \mathrm{tr} (a \nabla^2 v) + a^0=0 \quad\text{in }\Omega[r],
\end{equation}
where the matrix $a$ and the function $a^0$ are given by
\begin{align*}
    a(x)&=\tilde{a}(\nabla v(x)) = |\nabla v(x)|^{p-2} I_N + (p-2) |\nabla v(x)|^{p-4} \nabla v(x) \otimes \nabla v(x) \quad\text{for all }x \in\Omega[r],\\
    a^0(x)&= \tilde{a}^0(v(x),\nabla v(x))= (p-1)|\nabla v(x)|^p + N e^{-(p-1)v(x)} \quad\text{for all }x \in\Omega[r].
\end{align*}
The matrix $a=(a^{ij})_{i,j=1}^N$ is uniformly elliptic in $\Omega[r]$. Indeed, for every $\zeta\in\R^N$,
\[
    \langle a(x) \zeta , \zeta \rangle =|\nabla v(x)|^{p-2}|\zeta|^2+(p-2)|\nabla v(x)|^{p-4}\langle \nabla v(x),\zeta\rangle^2 \quad\text{for all }x \in\Omega[r],
\]
and therefore, by \eqref{nablav-bounds} and Cauchy-Schwartz,
\begin{equation*}
    m^{p-2}|\zeta|^2\le \langle a(x) \zeta , \zeta \rangle \le (p-1)M^{p-2}|\zeta|^2 \qquad \text{for all } x\in \Omega[r] \text{ and } \zeta\in\R^N.
\end{equation*}
Thus, one can take
\[
    \lambda=m^{p-2},\qquad\mu=(p-1)\left(\frac{M}{m}\right)^{p-2}.
\]
By \eqref{nablav-bounds} and by $v\in C^{1,\tilde{\alpha}}(\overline\Omega)$, we deduce that the functions $a^{ij}$ and $a^0$ belong to  $C^{0,\tilde{\alpha}}(\overline{\Omega[r]})$ for every $i,j \in \{1,\dots,N\}$. Moreover, recalling the definition of $\delta$ in \eqref{eq d and delta} and using the fact that $\delta(x) \leq \mathrm{diam}\,\Omega[r]$, we have that the zero-th order term satisfies 
\[
    |a^0|_0^{(2)} \leq (\mathrm{diam}\,\Omega[r])^2 \|a^0\|_{C^0(\overline{\Omega[r]})},
\]
and
\[
    [a^0]_{\tilde{\alpha}}^{(2)}\le (\mathrm{diam}\,\Omega[r])^{2+\tilde{\alpha}} \|a^0\|_{C^{0,\tilde{\alpha}}(\overline{\Omega[r]})}.
\]
Therefore, $a^0 \in C_{(2)}^{0,\tilde{\alpha}}(\Omega[r])$. Hence, there exist $A_1$ such that the coefficients bound \eqref{teor-L1} of Theorem \ref{lieberman-teor} holds.

Next, we rewrite the boundary condition \eqref{boundary-v} in a form that is uniformly oblique (cf. \eqref{teor-L2}) for all $q \in B_K(0)$ for a suitable $K \in (0,+\infty)$. Let $K>M$ and choose $\ell\in C^\infty([0,+\infty))$ such that
\begin{equation}\label{eq bound ell}
\begin{aligned}
    \ell(t)&=t^{(p-2)/2} &&\text{for }t\in[m^2/2,4M^2],\\
    \ell(t)&\ge c_\ell>0 &&\text{for }t\in[0,K^2],\\
    \ell'(t)&\ge0 &&\text{for }t\in[0,K^2].  
\end{aligned}    
\end{equation}
Define
\[
    \Phi(q)=\ell(|q|^2)q \quad \text{for } q \in B_K(0),\qquad b(x,q)=\langle \Phi(q),\bar{\nu}(x)\rangle-\beta \quad \text{for } (x,q) \in \Sigma[r] \times B_K(0).
\]
By \eqref{nablav-bounds} and \eqref{eq bound ell}, condition \eqref{boundary-v} is equivalent to
\begin{equation}\label{boundary-cond}
    b(x,\nabla v(x))=0 \quad\text{for all } x \in \Sigma[r].
\end{equation}
Furthermore, chain rule and \eqref{eq bound ell} yield for all $x \in \Sigma[r]$ and $q \in B_K(0)$
\[
    \langle \nabla_q b(x,q) \bar{\nu}(x), \bar{\nu}(x) \rangle
    =\langle \nabla \Phi(q)\bar{\nu}(x),\bar{\nu}(x)\rangle 
    =\ell(|q|^2)+2\ell'(|q|^2)\langle q,\bar{\nu}(x)\rangle^2
    \ge c_\ell.
\]
Hence, we conclude that \eqref{teor-L2} holds for $\chi = c_\ell$. 

We now establish the validity of \eqref{teor-L3} for a suitable constant $C$. Assume to extend $\bar{\nu}$ as in Lemma \ref{lem:tangential} in the computation below. Then,
\[
\nabla b(x,q) = \left(\nabla_x b(x,q) , \nabla_q b(x,q) \right) = \left( \nabla \bar{\nu} (x) \Phi(q) , \nabla \Phi(q) \bar{\nu}(x) \right).
\]
Straightforward algebraic manipulations yield
\begin{equation*}
\begin{split}
    |\nabla_x b(x,q) - \nabla_x b(y,t)| &\leq | \nabla \bar{\nu} (x) -  \nabla \bar{\nu} (y) | |\Phi(q)| + |\nabla \bar{\nu} (y) | |\Phi(q) - \Phi(t)|,
    \\
    |\nabla_q b(x,q) - \nabla_q b(y,t)| &\leq |\nabla \Phi(q) - \nabla \Phi(t)| |\bar{\nu}(x)| + |\nabla \Phi(t)| |\bar{\nu}(x) - \bar{\nu}(y)|,
\end{split}
\end{equation*}
and since $\bar{\nu} \in C^{1,\tilde{\alpha}}(\partial\Omega)$ and $\Phi\in C^\infty(B_K(0), \R^N)$, we conclude that 
\[
    |\nabla b(x,q)- \nabla b(y,t)| \le C\left(|q-t|^{\tilde{\alpha}} + |x-y|^{\tilde{\alpha}} \right) \quad \text{for } x,y \in \Sigma[R],\, |x-y| \le \frac{1}{2} d^\ast(x) \text{ and } q,t \in B_K(0),
\]
where we have used the fact that $|q-t| = |q-t|^{1-\tilde{\alpha}} |q-t|^{\tilde{\alpha}} \leq (2K)^{1-\tilde{\alpha}} |q-t|^{\tilde{\alpha}}$.

After possibly reducing $r$ so that $d^\ast(x) \le \mathrm{diam}\,\Omega[r] \le 1$ (hence $d^\ast(x)^{-1-\tilde{\alpha}} \ge 1$), this implies  that \eqref{teor-L3} holds. Therefore, Theorem \ref{lieberman-teor}, applied to problem \eqref{eq-nondiv}-\eqref{boundary-cond}, implies
\[
    v\in C^{2,\tilde{\alpha}}(\overline{\Omega[r/2]}).
\]

Finally, since $\partial\Omega$ is compact, a standard covering argument yields the conclusion that $v$ has the desired regularity in the collar $\Omega_\varepsilon$ for a suitable $\varepsilon>0$. Using $u=-e^v$, the validity of the statement follows. 
\end{proof}

\section{Shape derivative and overdetermined condition}\label{sec: appendix shape}
In this section, we compute the overdetermined condition as stationary condition of the following energy functional 
\begin{equation}\label{appendix:energy}
    \mathcal{F}_p(u)=\frac{1}{p}\int_\Omega|\nabla u|^p\d{x}+\frac{\beta}{p}\int_{\partial\Omega}u^p\d{}\mathcal{H}^{n-1}+N\int_\Omega u\d{x}.
\end{equation}
following the standard approach presented in \cite{hepi18} (see \cite{bw} for Robin boundary conditions for the functional $\mathcal{F}_2$).
It is well known that the Euler-Lagrange equation of $\mathcal{F}_p$ in \eqref{appendix:energy} is exactly problem \eqref{eq:problem_p}. In particular, one has that 
\begin{equation*}
    \mathcal{F}_p(\Omega) = \min_{v \in W^{1,p}(\Omega)} \mathcal{F}_p(v).
\end{equation*}

Let $V\in C^{2,\alpha}(\R^n,\R^n)$ be a smooth vector field. For a given $t_0>0$, we denote by $\{\Omega_t\}_{0\leq t<t_0}$ a family of perturbations of the domain $\Omega$ of the form
\[\Omega_t=(I_N+tV)\Omega.\]
If $t_0$ is sufficiently small, the function 
\begin{equation}\label{eq: app Phi}
    \Phi(t,x)=I_N(x)+tV(x), \quad \text{for all }t \in [0,t_0),\,  x \in \Omega,
\end{equation} is bi-Lipschitz homeomorphism (see \cite[Remark 2.2]{cnt2024}). In particular, the shape derivative of the functional $\mathcal{F}_p$ at a given fix domain $\Omega$ in the direction $V$ is defined as the following limit (if such limit exists)
\begin{equation*}
    d\mathcal{F}_p(\Omega,V):=\lim_{t \to 0^+}\frac{\mathcal{F}_p(\Omega_t)-\mathcal{F}_p(\Omega)}{t}.
\end{equation*}
In order to compute the shape derivative of the functional $\mathcal{F}_p$ in the smooth setting we are interested in, we recall the following formulae of \cite[Theorem 5.2.2 \& Theorem 5.4.17]{hepi18} (see also \cite[Corollary 5.2.5 \& Proposition 5.4.18]{hepi18})

\begin{teo}[Hadamard Formulae]\label{thm hadamard}
    Let $\Omega$ be an open, bounded connected subset of $\R^n$ of class $C^{2,\alpha}$ with $\alpha \in (0,1)$. Let $\Phi(t,x)$ be as in \eqref{eq: app Phi}. Let $t \in [0,t_0) \to f(t) \in L^1(\Omega_t)$ and $t \in [0,t_0) \to g(t) \in L^1(\Omega_t)$. Assume that 
    \begin{equation*}
    \begin{split}
    &t \in [0,t_0) \to F(t) = f(t, \Phi(t,x)) \in L^1(\Omega) \text{ is differentiable at $0$ and } f(0) \in W^{1,1}(\Omega),
    \\
    &t \in [0,t_0) \to G(t) = g(t, \Phi(t,x)) \in W^{1,1}(\Omega) \text{ is differentiable at $0$ and } g(0) \in W^{2,1}(\Omega).
    \end{split}
    \end{equation*}
    Then,
    \begin{align}
        \frac{d}{dt}\int_{\Omega_t}f(t)\d{x}\bigg\rvert_{t=0}&=\int_\Omega f'(0)\d{x}+\int_{\partial\Omega}f(0)\langle V,\nu\rangle\d{}\mathcal{H}^{n-1},\label{hadamard-interior}\\
        \frac{d}{dt}\int_{\partial\Omega_t}g(t)\d{}\mathcal{H}^{n-1}\bigg\rvert_{t=0}&=\int_{\partial\Omega}g'(0)+\left(g_\nu(0)+(N-1)\mathrm{M_{\partial\Omega}} \,g(0)\right)\langle V,\nu\rangle\d{}\mathcal{H}^{n-1}.\label{hadamard-boundary}
    \end{align}
\end{teo}
Moreover, the perturbation $\Phi(t,x)$ is called volume preserving if
\begin{equation}\label{eq app volume preserving}
    \int_{\partial\Omega}\langle V,\nu\rangle\d{}\mathcal{H}^{n-1}=0.
\end{equation}
We are now ready to derive the condition \eqref{intro eq overdetermined cond p}. We just sketch the proof (cf. \cite[Thm. 4.1]{bw}).
\begin{teo}
    Let $\Omega$ be an open, bounded connected subset of $\R^n$ of class $C^{2,\alpha}$ with $\alpha \in (0,1)$. Let $\mathcal{F}_p$ be as in \eqref{appendix:energy} and let $\Phi$ be given by \eqref{eq: app Phi}. Assume that $\Phi$ is a volume preserving homeomorphism. Then, the first order shape derivative of $\mathcal{F}_p$ is given by
    \begin{equation*}
        d\mathcal{F}_p(\Omega,V)=\frac{1}{p}\int_{\partial\Omega}\left(|\nabla u|^p-p|\nabla u|^{p-2}(u_\nu)^2+\beta(N-1)\mathrm{M}_{\partial\Omega}u^p+pNu\right)\langle V,\nu\rangle\d{}\mathcal{H}^{n-1}.
    \end{equation*}
    In particular, $\Omega$ is a critical set for volume preserving perturbation for the functional $\mathcal{F}_p$, i.e. 
    \[
    d\mathcal{F}_p(\Omega,V)=0,
    \]
    if and only if
    \begin{equation}\label{eq app overdetermined}
    |\nabla u|^p-p|\nabla u|^{p-2}(u_\nu)^2+\beta(N-1)\mathrm{M}_{\partial\Omega}u^p+pNu=C,
    \end{equation}
    for some constant $C\in\R$.

\end{teo}

\begin{proof}
    Let $\tilde{u}(t, \cdot)$ be the weak solution of problem \eqref{eq:problem_p} for $\Omega = \Omega_t$. Notice that $\tilde{u}(0,\cdot) = u$, where $u$ is the unique weak solution of \eqref{eq:problem_p} in $\Omega$. Define
    \[
    u(t) = \tilde{u}(t, \Phi(t,x)) \quad \text{for all } t\in[0,t_0),\, x\in\Omega.
    \] 
    We derive each term of $\mathcal{F}_p$. For the differentiability assumptions required by Theorem \ref{thm hadamard} we refer for instance to \cite[Section 5.3]{hepi18} for Dirichlet boundary conditions and to \cite[Lemma 3.1]{crist2026} for Robin boundary conditions. Moreover, for the smoothness of $u(t)$ we refer to Proposition \ref{prop properties u} and Theorem \ref{thm regularity C2} applied to the domain $\Omega_t$.  By formula \eqref{hadamard-interior} we have
    \begin{equation}\label{interior-deriv}
        \begin{aligned}
            \frac{d}{dt}\left( \int_{\Omega_t}|\nabla u(t)|^p \d{x} \right) \bigg\rvert_{t=0} &=\int_\Omega p|\nabla u|^{p-2}\langle\nabla u,\nabla \dot{u}\rangle\d{x}+\int_{\partial\Omega}|\nabla u|^p\langle V,\nu\rangle\d{}\mathcal{H}^{n-1},\\
            \frac{d}{dt}\left( \int_{\Omega_t}N u(t) \d{x}\right) \bigg\rvert_{t=0} &=N\int_\Omega \dot{u} \d{x}+\int_{\partial\Omega}Nu\langle V,\nu\rangle\d{}\mathcal{H}^{n-1},
        \end{aligned}
    \end{equation}
    where $\dot{u} = \frac{d}{dt}u(t)\big\rvert_{t=0}$ is a function in $W^{1,p}(\Omega)$ (for the precise definition and the problem solved by $\dot{u}$, see \cite[Section 5.3]{hepi18}). Similarly, using \eqref{hadamard-boundary} for the boundary term
    \begin{equation}\label{boundary-deriv}
        \frac{d}{dt}\int_{\partial\Omega_t}|u(t)|^p\d{}\mathcal{H}^{n-1}\bigg\rvert_{t=0}=\int_{\partial\Omega}p|u|^{p-2}u \dot{u} + \left(p|u|^{p-2} u u_\nu+(N-1) \mathrm{M}_{\partial\Omega} \, u^p \right) \langle V,\nu\rangle \d{}\mathcal{H}^{n-1}.
    \end{equation}
    Hence, combining \eqref{interior-deriv} and \eqref{boundary-deriv} yields
    \begin{equation*}
        \begin{aligned}
            d\mathcal{F}_p(\Omega,V)=&\int_\Omega \left( |\nabla u|^{p-2}\langle\nabla u,\nabla \dot{u} \rangle + N \dot{u} \right) \d{x} + \beta\int_{\partial\Omega}|u|^{p-2}u\dot{u}\d{}\mathcal{H}^{n-1}
            \\
            &+\int_{\partial\Omega}\langle V,\nu\rangle\left(\frac{|\nabla u|^p}{p}+\beta|u|^{p-2}uu_\nu+\frac{\beta}{p}(N-1)\mathrm{M}_{\partial\Omega}u^p+Nu\right)\d{}\mathcal{H}^{n-1}
            \\
            =&\frac{1}{p}\int_{\partial\Omega}\left(|\nabla u|^p-p|\nabla u|^{p-2}(u_\nu)^2+\beta(N-1)\mathrm{M}_{\partial\Omega}u^p+pNu\right)\langle V,\nu\rangle\d{}\mathcal{H}^{n-1},
        \end{aligned}
    \end{equation*}
    where the first term vanishes because is exactly the weak formulation of \eqref{eq:problem_p} for $u$ tested for $\varphi = \dot{u}$. 

    Finally, \eqref{eq app overdetermined} follows by the same argument of \cite[Theorem 4.1]{bw}, once \eqref{eq app volume preserving} is in force (i.e. $V$ is volume preserving). The proof is complete.
\end{proof}

\subsection*{Acknowledgments}
This work has been partially supported by GNAMPA of INdAM. 

\noindent
Riccardo Molinarolo is supported by the Project ``Giochi a campo medio, trasporto e ottimizzazione in sistemi auto-organizzati e machine learning”, funded by MUR, D.D. 47/2025, PNRR - Missione 4, Componente 2, Investimento 1.2 - funded by European Union NextGenerationEU, CUP B33C25000380001.

\noindent 
The authors are indebted to Professors N.\,Gavitone and C.\,Trombetti for encouraging us to investigate this problem and for helpful discussions and valuable suggestions. 

\addtocontents{toc}{\protect\setcounter{tocdepth}{1}}
\subsection*{AI statement}
No generative AI tool was used to generate research ideas, interpret
results, or write scientific content.

\bibliographystyle{plain}
\bibliography{bibliography}

@article {bk2011,
    AUTHOR = {Buttazzo, Giuseppe and Kawohl, Bernd},
     TITLE = {Overdetermined boundary value problems for the
              {$\infty$}-{L}aplacian},
   JOURNAL = {Int. Math. Res. Not. IMRN},
  FJOURNAL = {International Mathematics Research Notices. IMRN},
      YEAR = {2011},
    VOLUME = {2011},
    NUMBER = {2},
     PAGES = {237--247},
      ISSN = {1073-7928,1687-0247},
   MRCLASS = {35N25 (35J60 35J92)},
  MRNUMBER = {2764863},
MRREVIEWER = {Steven\ George\ Krantz},
       DOI = {10.1093/imrn/rnq071},
       URL = {https://doi.org/10.1093/imrn/rnq071},
}

@book {giusti84,
    AUTHOR = {Giusti, Enrico},
     TITLE = {Minimal surfaces and functions of bounded variation},
    SERIES = {Monographs in Mathematics},
    VOLUME = {80},
 PUBLISHER = {Birkh\"auser Verlag, Basel},
      YEAR = {1984},
     PAGES = {xii+240},
      ISBN = {0-8176-3153-4},
   MRCLASS = {58E12 (49F10 53A10)},
  MRNUMBER = {775682},
MRREVIEWER = {Helmut\ Kaul},
       DOI = {10.1007/978-1-4684-9486-0},
       URL = {https://doi.org/10.1007/978-1-4684-9486-0},
}

@book {li12,
    AUTHOR = {Li, Peter},
     TITLE = {Geometric analysis},
    SERIES = {Cambridge Studies in Advanced Mathematics},
    VOLUME = {134},
 PUBLISHER = {Cambridge University Press, Cambridge},
      YEAR = {2012},
     PAGES = {x+406},
      ISBN = {978-1-107-02064-1},
   MRCLASS = {58-02 (35P15 53C21 58J32 58J35)},
  MRNUMBER = {2962229},
MRREVIEWER = {Fr\'ed\'eric\ Robert},
       DOI = {10.1017/CBO9781139105798},
       URL = {https://doi.org/10.1017/CBO9781139105798},
}

@book {hepi18,
    AUTHOR = {Henrot, Antoine and Pierre, Michel},
     TITLE = {Shape variation and optimization},
    SERIES = {EMS Tracts in Mathematics},
    VOLUME = {28},
 PUBLISHER = {European Mathematical Society (EMS), Z\"urich},
      YEAR = {2018},
     PAGES = {xi+365},
      ISBN = {978-3-03719-178-1},
   MRCLASS = {49Q10 (31B15 35J20 35R35 49-02 58E25 65J05)},
  MRNUMBER = {3791463},
MRREVIEWER = {Jan\ Soko\l owski},
       DOI = {10.4171/178},
       URL = {https://doi.org/10.4171/178},
}

@article {gamo26,
    AUTHOR = {Gavitone, Nunzia and Molinarolo, Riccardo},
     TITLE = {On the {S}errin's problem with {R}obin boundary conditions},
   JOURNAL = {Nonlinear Anal.},
  FJOURNAL = {Nonlinear Analysis. Theory, Methods \& Applications. An
              International Multidisciplinary Journal},
    VOLUME = {268},
      YEAR = {2026},
     PAGES = {Paper No. 114081, 10},
      ISSN = {0362-546X,1873-5215},
   MRCLASS = {35N25 (35A23 35B50 53A10)},
  MRNUMBER = {5035976},
       DOI = {10.1016/j.na.2026.114081},
       URL = {https://doi.org/10.1016/j.na.2026.114081},
}

@book {gitr83,
    AUTHOR = {Gilbarg, David and Trudinger, Neil S.},
     TITLE = {Elliptic partial differential equations of second order},
    SERIES = {Grundlehren der mathematischen Wissenschaften},
    VOLUME = {224},
   EDITION = {Second},
 PUBLISHER = {Springer-Verlag, Berlin},
      YEAR = {1983},
     PAGES = {xiii+513},
      ISBN = {3-540-13025-X},
   MRCLASS = {35Jxx (35-01)},
  MRNUMBER = {737190},
MRREVIEWER = {O.\ John},
       DOI = {10.1007/978-3-642-61798-0},
       URL = {https://doi.org/10.1007/978-3-642-61798-0},
}

@article {krpa81,
    AUTHOR = {Krantz, Steven G. and Parks, Harold R.},
     TITLE = {Distance to {$C\sp{k}$}\ hypersurfaces},
   JOURNAL = {J. Differential Equations},
  FJOURNAL = {Journal of Differential Equations},
    VOLUME = {40},
      YEAR = {1981},
    NUMBER = {1},
     PAGES = {116--120},
      ISSN = {0022-0396,1090-2732},
   MRCLASS = {58C07 (53A07)},
  MRNUMBER = {614221},
MRREVIEWER = {Jean-Marie\ Morvan},
       DOI = {10.1016/0022-0396(81)90013-9},
       URL = {https://doi.org/10.1016/0022-0396(81)90013-9},
}

@article {mapo19,
    AUTHOR = {Magnanini, Rolando and Poggesi, Giorgio},
     TITLE = {On the stability for {A}lexandrov's soap bubble theorem},
   JOURNAL = {J. Anal. Math.},
  FJOURNAL = {Journal d'Analyse Math\'ematique},
    VOLUME = {139},
      YEAR = {2019},
    NUMBER = {1},
     PAGES = {179--205},
      ISSN = {0021-7670,1565-8538},
   MRCLASS = {53A10 (35J93)},
  MRNUMBER = {4041100},
MRREVIEWER = {Jo\~ao\ Lucas Marques Barbosa},
       DOI = {10.1007/s11854-019-0058-y},
       URL = {https://doi.org/10.1007/s11854-019-0058-y},
}

@article {serrin,
    AUTHOR = {Serrin, James},
     TITLE = {A symmetry problem in potential theory},
   JOURNAL = {Arch. Rational Mech. Anal.},
  FJOURNAL = {Archive for Rational Mechanics and Analysis},
    VOLUME = {43},
      YEAR = {1971},
     PAGES = {304--318},
      ISSN = {0003-9527},
   MRCLASS = {31B20},
  MRNUMBER = {333220},
MRREVIEWER = {Gottfried\ Anger},
       DOI = {10.1007/BF00250468},
       URL = {https://doi.org/10.1007/BF00250468},
}

@book{PucciSerrin07,
  title={The maximum principle},
  author={Pucci, Patrizia and Serrin, James},
  year={2007},
  publisher={Springer}
}

@book{Lieberman12,
  title={Oblique derivative problems for elliptic equations},
  author={Lieberman, Gary M.},
  year={2013},
  publisher={World Scientific}
}

@article{Lieberman88,
  title = {Boundary regularity for solutions of degenerate elliptic equations},
  volume = {12},
  ISSN = {0362-546X},
  url = {http://dx.doi.org/10.1016/0362-546X(88)90053-3},
  DOI = {10.1016/0362-546x(88)90053-3},
  number = {11},
  journal = {Nonlinear Analysis: Theory,  Methods \&; Applications},
  publisher = {Elsevier BV},
  author = {Lieberman,  Gary M.},
  year = {1988},
  month = Nov,
  pages = {1203–1219}
}

@article {RuHuCh24,
    AUTHOR = {Ruan, Qihua and Huang, Qin and Chen, Fan},
     TITLE = {The {$p$}-{L}aplacian overdetermined problem on {R}iemannian
              manifolds},
   JOURNAL = {Ann. Mat. Pura Appl. (4)},
  FJOURNAL = {Annali di Matematica Pura ed Applicata. Series IV},
    VOLUME = {203},
      YEAR = {2024},
    NUMBER = {2},
     PAGES = {647--662},
      ISSN = {0373-3114,1618-1891},
   MRCLASS = {35A23 (35J92 35N25)},
  MRNUMBER = {4715314},
MRREVIEWER = {Stefan\ Steinerberger},
       DOI = {10.1007/s10231-023-01377-0},
       URL = {https://doi.org/10.1007/s10231-023-01377-0},
}

@article {CoFe20,
    AUTHOR = {Colasuonno, Francesca and Ferrari, Fausto},
     TITLE = {The soap bubble-theorem and a {$p$}-{L}aplacian overdetermined
              problem},
   JOURNAL = {Commun. Pure Appl. Anal.},
  FJOURNAL = {Communications on Pure and Applied Analysis},
    VOLUME = {19},
      YEAR = {2020},
    NUMBER = {2},
     PAGES = {983--1000},
      ISSN = {1534-0392,1553-5258},
   MRCLASS = {35N25 (35A23 35B06 35J92 53A10)},
  MRNUMBER = {4043775},
MRREVIEWER = {Bogdan\ Rai\c t\u a},
       DOI = {10.3934/cpaa.2020045},
       URL = {https://doi.org/10.3934/cpaa.2020045},
}

@article {abr1999,
    AUTHOR = {Aftalion, Amandine and Busca, J\'er\^ome and Reichel,
              Wolfgang},
     TITLE = {Approximate radial symmetry for overdetermined boundary value
              problems},
   JOURNAL = {Adv. Differential Equations},
  FJOURNAL = {Advances in Differential Equations},
    VOLUME = {4},
      YEAR = {1999},
    NUMBER = {6},
     PAGES = {907--932},
      ISSN = {1079-9389},
   MRCLASS = {35J65 (35B50)},
  MRNUMBER = {1729395},
MRREVIEWER = {Alexandra\ Kurepa},
}

@article {a58,
    AUTHOR = {Aleksandrov, Aleksandr D.},
     TITLE = {Uniqueness theorems for surfaces in the large. {V}},
   JOURNAL = {Vestnik Leningrad. Univ.},
  FJOURNAL = {Vestnik Leningrad. Univ.},
    VOLUME = {13},
      YEAR = {1958},
    NUMBER = {19},
     PAGES = {5--8},
   MRCLASS = {53.00},
  MRNUMBER = {102114},
MRREVIEWER = {H.\ Busemann},
}

@article {a62,
    AUTHOR = {Alexandrov, Aleksandr D.},
     TITLE = {A characteristic property of spheres},
   JOURNAL = {Ann. Mat. Pura Appl. (4)},
  FJOURNAL = {Annali di Matematica Pura ed Applicata. Serie Quarta},
    VOLUME = {58},
      YEAR = {1962},
     PAGES = {303--315},
      ISSN = {0003-4622},
   MRCLASS = {53.75},
  MRNUMBER = {143162},
MRREVIEWER = {A.\ Fialkow},
       DOI = {10.1007/BF02413056},
       URL = {https://doi.org/10.1007/BF02413056},
}

@article {bw,
    AUTHOR = {Bandle, Catherine and Wagner, Alfred},
     TITLE = {Second domain variation for problems with {R}obin boundary
              conditions},
   JOURNAL = {J. Optim. Theory Appl.},
  FJOURNAL = {Journal of Optimization Theory and Applications},
    VOLUME = {167},
      YEAR = {2015},
    NUMBER = {2},
     PAGES = {430--463},
      ISSN = {0022-3239,1573-2878},
   MRCLASS = {49Q10 (35J20 35N25 49J20 49K20 49R05)},
  MRNUMBER = {3412445},
MRREVIEWER = {Andrzej\ M.\ My\'sli\'nski},
       DOI = {10.1007/s10957-015-0801-1},
       URL = {https://doi.org/10.1007/s10957-015-0801-1},
}

@article {bc,
    AUTHOR = {Bianchini, Chiara and Ciraolo, Giulio},
     TITLE = {Wulff shape characterizations in overdetermined anisotropic
              elliptic problems},
   JOURNAL = {Comm. Partial Differential Equations},
  FJOURNAL = {Communications in Partial Differential Equations},
    VOLUME = {43},
      YEAR = {2018},
    NUMBER = {5},
     PAGES = {790--820},
      ISSN = {0360-5302,1532-4133},
   MRCLASS = {35N25 (35A23 35B06 35J25)},
  MRNUMBER = {3920523},
MRREVIEWER = {Steven\ George\ Krantz},
       DOI = {10.1080/03605302.2018.1475488},
       URL = {https://doi.org/10.1080/03605302.2018.1475488},
}

@article {bnst,
   author={Brandolini, Barbara and Nitsch, Carlo and Salani, Paolo and Trombetti, Cristina},
   title={Serrin-type overdetermined problems: an alternative proof},
   journal={Arch. Ration. Mech. Anal.},
   volume={190},
   year={2008},
   number={2},
   pages={267--280},
   issn={0003-9527},
   review={\MR{2448319}},
   doi={10.1007/s00205-008-0119-3},
   url={https://doi.org/10.1007/s00205-008-0119-3},
}

@article {bnst08,
    AUTHOR = {Brandolini, Barbara and Nitsch, Carlo and Salani, Paolo and Trombetti, Cristina},
     TITLE = {On the stability of the {S}errin problem},
   JOURNAL = {J. Differential Equations},
  FJOURNAL = {Journal of Differential Equations},
    VOLUME = {245},
      YEAR = {2008},
    NUMBER = {6},
     PAGES = {1566--1583},
      ISSN = {0022-0396,1090-2732},
   MRCLASS = {35J25 (35B35)},
  MRNUMBER = {2436453},
MRREVIEWER = {Antoine\ Henrot},
       DOI = {10.1016/j.jde.2008.06.010},
       URL = {https://doi.org/10.1016/j.jde.2008.06.010},
}

@article{bgnt,
   author={Brandolini, Barbara and Gavitone, Nunzia and Nitsch, Carlo and Trombetti, Cristina},
   title={Characterization of ellipsoids through an overdetermined boundary
   value problem of Monge-Amp\`ere type},
   language={English, with English and French summaries},
   journal={J. Math. Pures Appl. (9)},
   volume={101},
   year={2014},
   number={6},
   pages={828--841},
   issn={0021-7824},
   review={\MR{3205644}},
   doi={10.1016/j.matpur.2013.10.005},
   url={https://doi.org/10.1016/j.matpur.2013.10.005},
}

@article{bh,
   author={Brock, Friedemann and Henrot, Antoine},
   title={A symmetry result for an overdetermined elliptic problem using
   continuous rearrangement and domain derivative},
   journal={Rend. Circ. Mat. Palermo (2)},
   volume={51},
   year={2002},
   number={3},
   pages={375--390},
   issn={0009-725X},
   review={\MR{1947461}},
   doi={10.1007/BF02871848},
   url={https://doi.org/10.1007/BF02871848}
}

@article {cgs,
    AUTHOR = {Caffarelli, Luis and Garofalo, Nicola and Seg\`ala, Fausto},
     TITLE = {A gradient bound for entire solutions of quasi-linear
              equations and its consequences},
   JOURNAL = {Comm. Pure Appl. Math.},
  FJOURNAL = {Communications on Pure and Applied Mathematics},
    VOLUME = {47},
      YEAR = {1994},
    NUMBER = {11},
     PAGES = {1457--1473},
      ISSN = {0010-3640,1097-0312},
   MRCLASS = {35B45 (35J60 53A10)},
  MRNUMBER = {1296785},
MRREVIEWER = {C.\ A.\ Swanson},
       DOI = {10.1002/cpa.3160471103},
       URL = {https://doi.org/10.1002/cpa.3160471103},
}

@article {cs,
    AUTHOR = {Cianchi, Andrea and Salani, Paolo},
     TITLE = {Overdetermined anisotropic elliptic problems},
   JOURNAL = {Math. Ann.},
  FJOURNAL = {Mathematische Annalen},
    VOLUME = {345},
      YEAR = {2009},
    NUMBER = {4},
     PAGES = {859--881},
      ISSN = {0025-5831,1432-1807},
   MRCLASS = {35N10 (35J35)},
  MRNUMBER = {2545870},
       DOI = {10.1007/s00208-009-0386-9},
       URL = {https://doi.org/10.1007/s00208-009-0386-9},
}

@article {cm2017,
    AUTHOR = {Ciraolo, Giulio and Maggi, Francesco},
     TITLE = {On the shape of compact hypersurfaces with almost-constant
              mean curvature},
   JOURNAL = {Comm. Pure Appl. Math.},
  FJOURNAL = {Communications on Pure and Applied Mathematics},
    VOLUME = {70},
      YEAR = {2017},
    NUMBER = {4},
     PAGES = {665--716},
      ISSN = {0010-3640,1097-0312},
   MRCLASS = {53A10 (35J93 52A20 53C42)},
  MRNUMBER = {3628882},
MRREVIEWER = {Fei-Tsen\ Liang},
       DOI = {10.1002/cpa.21683},
       URL = {https://doi.org/10.1002/cpa.21683},
}

@article {cms2015,
    AUTHOR = {Ciraolo, Giulio and Magnanini, Rolando and Sakaguchi, Shigeru},
     TITLE = {Symmetry of minimizers with a level surface parallel to the
              boundary},
   JOURNAL = {J. Eur. Math. Soc. (JEMS)},
  FJOURNAL = {Journal of the European Mathematical Society (JEMS)},
    VOLUME = {17},
      YEAR = {2015},
    NUMBER = {11},
     PAGES = {2789--2804},
      ISSN = {1435-9855,1435-9863},
   MRCLASS = {35A15 (35B06 35J20 35J60 35K55 35N25 49J10 49Q10)},
  MRNUMBER = {3420522},
MRREVIEWER = {Pablo\ Pedregal},
       DOI = {10.4171/JEMS/571},
       URL = {https://doi.org/10.4171/JEMS/571},
}

@article {cmv,
    AUTHOR = {Ciraolo, Giulio and Magnanini, Rolando and Vespri, Vincenzo},
     TITLE = {H\"older stability for {S}errin's overdetermined problem},
   JOURNAL = {Ann. Mat. Pura Appl. (4)},
  FJOURNAL = {Annali di Matematica Pura ed Applicata. Series IV},
    VOLUME = {195},
      YEAR = {2016},
    NUMBER = {4},
     PAGES = {1333--1345},
      ISSN = {0373-3114,1618-1891},
   MRCLASS = {35N25 (35B06 35B09 35B35 35J91)},
  MRNUMBER = {3522349},
       DOI = {10.1007/s10231-015-0518-7},
       URL = {https://doi.org/10.1007/s10231-015-0518-7},
}

@article {cv18,
    AUTHOR = {Ciraolo, Giulio and Vezzoni, Luigi},
     TITLE = {A sharp quantitative version of {A}lexandrov's theorem via the
              method of moving planes},
   JOURNAL = {J. Eur. Math. Soc. (JEMS)},
  FJOURNAL = {Journal of the European Mathematical Society (JEMS)},
    VOLUME = {20},
      YEAR = {2018},
    NUMBER = {2},
     PAGES = {261--299},
      ISSN = {1435-9855,1435-9863},
   MRCLASS = {53A10 (35B35 35B50 35B51 35J93 53C21)},
  MRNUMBER = {3760295},
MRREVIEWER = {Martin\ L. P. Kilian},
       DOI = {10.4171/JEMS/766},
       URL = {https://doi.org/10.4171/JEMS/766},
}

@article {cv2019,
    AUTHOR = {Ciraolo, Giulio and Vezzoni, Luigi},
     TITLE = {On {S}errin's overdetermined problem in space forms},
   JOURNAL = {Manuscripta Math.},
  FJOURNAL = {Manuscripta Mathematica},
    VOLUME = {159},
      YEAR = {2019},
    NUMBER = {3-4},
     PAGES = {445--452},
      ISSN = {0025-2611,1432-1785},
   MRCLASS = {35N25 (35B50 35R01 53C24 58J05)},
  MRNUMBER = {3959271},
MRREVIEWER = {Alberto\ Parmeggiani},
       DOI = {10.1007/s00229-018-1079-z},
       URL = {https://doi.org/10.1007/s00229-018-1079-z},
}

@article {dp,
    AUTHOR = {Damascelli, Lucio and Pacella, Filomena},
     TITLE = {Monotonicity and symmetry results for {$p$}-{L}aplace
              equations and applications},
   JOURNAL = {Adv. Differential Equations},
  FJOURNAL = {Advances in Differential Equations},
    VOLUME = {5},
      YEAR = {2000},
    NUMBER = {7-9},
     PAGES = {1179--1200},
      ISSN = {1079-9389},
   MRCLASS = {35J65 (35B05 35B50 35J70)},
  MRNUMBER = {1776351},
MRREVIEWER = {Srinivasan\ Kesavan},
}

@article {dep23,
    AUTHOR = {Dom\'inguez-V\'azquez, Miguel and Enciso, Alberto and
              Peralta-Salas, Daniel},
     TITLE = {Overdetermined boundary problems with nonconstant {D}irichlet
              and {N}eumann data},
   JOURNAL = {Anal. PDE},
  FJOURNAL = {Analysis \& PDE},
    VOLUME = {16},
      YEAR = {2023},
    NUMBER = {9},
     PAGES = {1989--2003},
      ISSN = {2157-5045,1948-206X},
   MRCLASS = {35N25 (35J61)},
  MRNUMBER = {4668085},
       DOI = {10.2140/apde.2023.16.1989},
       URL = {https://doi.org/10.2140/apde.2023.16.1989},
}

@article {fm2015,
    AUTHOR = {Fall, Mouhamed Moustapha and Minlend, Ignace Aristide},
     TITLE = {Serrin's over-determined problem on {R}iemannian manifolds},
   JOURNAL = {Adv. Calc. Var.},
  FJOURNAL = {Advances in Calculus of Variations},
    VOLUME = {8},
      YEAR = {2015},
    NUMBER = {4},
     PAGES = {371--400},
      ISSN = {1864-8258,1864-8266},
   MRCLASS = {58J05 (35N25 35R01 58J32 58J37)},
  MRNUMBER = {3403432},
MRREVIEWER = {Jyotshana\ Prajapat},
       DOI = {10.1515/acv-2014-0017},
       URL = {https://doi.org/10.1515/acv-2014-0017},
}

@article {fj,
    AUTHOR = {Fall, Mouhamed Moustapha and Jarohs, Sven},
     TITLE = {Overdetermined problems with fractional {L}aplacian},
   JOURNAL = {ESAIM Control Optim. Calc. Var.},
  FJOURNAL = {ESAIM. Control, Optimisation and Calculus of Variations},
    VOLUME = {21},
      YEAR = {2015},
    NUMBER = {4},
     PAGES = {924--938},
      ISSN = {1292-8119,1262-3377},
   MRCLASS = {35N25 (35B50 35R11)},
  MRNUMBER = {3395749},
       DOI = {10.1051/cocv/2014048},
       URL = {https://doi.org/10.1051/cocv/2014048},
}

@article {fk,
    AUTHOR = {Farina, Alberto and Kawohl, Bernd},
     TITLE = {Remarks on an overdetermined boundary value problem},
   JOURNAL = {Calc. Var. Partial Differential Equations},
  FJOURNAL = {Calculus of Variations and Partial Differential Equations},
    VOLUME = {31},
      YEAR = {2008},
    NUMBER = {3},
     PAGES = {351--357},
      ISSN = {0944-2669,1432-0835},
   MRCLASS = {35N10 (35J65)},
  MRNUMBER = {2366129},
       DOI = {10.1007/s00526-007-0115-8},
       URL = {https://doi.org/10.1007/s00526-007-0115-8},
}

@article {fr,
    AUTHOR = {Farina, Alberto and Roncoroni, Alberto},
     TITLE = {Serrin's type problems in warped product manifolds},
   JOURNAL = {Commun. Contemp. Math.},
  FJOURNAL = {Communications in Contemporary Mathematics},
    VOLUME = {24},
      YEAR = {2022},
    NUMBER = {4},
     PAGES = {Paper No. 2150020, 21},
      ISSN = {0219-1997,1793-6683},
   MRCLASS = {35R01 (35B50 35N25 53C24 58J05 58J32)},
  MRNUMBER = {4414163},
       DOI = {10.1142/S0219199721500206},
       URL = {https://doi.org/10.1142/S0219199721500206},
}

@article {fv,
    AUTHOR = {Farina, Alberto and Valdinoci, Enrico},
     TITLE = {A pointwise gradient estimate in possibly unbounded domains
              with nonnegative mean curvature},
   JOURNAL = {Adv. Math.},
  FJOURNAL = {Advances in Mathematics},
    VOLUME = {225},
      YEAR = {2010},
    NUMBER = {5},
     PAGES = {2808--2827},
      ISSN = {0001-8708,1090-2082},
   MRCLASS = {35J91 (35B45 35B65 35J25)},
  MRNUMBER = {2680184},
MRREVIEWER = {Siegfried\ Carl},
       DOI = {10.1016/j.aim.2010.05.008},
       URL = {https://doi.org/10.1016/j.aim.2010.05.008},
}

@article {fz2025,
    AUTHOR = {Figalli, Alessio and Zhang, Yi Ru-Ya},
     TITLE = {Serrin’s overdetermined problem in rough domains},
   JOURNAL = {J. Eur. Math. Soc. (JEMS)},
  FJOURNAL = {Journal of the European Mathematical Society (JEMS)},
      YEAR = {2025},
       DOI = {10.4171/JEMS/1726},
       URL = {https://doi.org/10.4171/JEMS/1726},
}

@article {f,
    AUTHOR = {Fragal\`a, Ilaria},
     TITLE = {Symmetry results for overdetermined problems on convex domains
              via {B}runn-{M}inkowski inequalities},
   JOURNAL = {J. Math. Pures Appl. (9)},
  FJOURNAL = {Journal de Math\'ematiques Pures et Appliqu\'ees. Neuvi\`eme
              S\'erie},
    VOLUME = {97},
      YEAR = {2012},
    NUMBER = {1},
     PAGES = {55--65},
      ISSN = {0021-7824,1776-3371},
   MRCLASS = {35N25 (35B06 35J57 49Q10 52A20 52A40)},
  MRNUMBER = {2863764},
MRREVIEWER = {Nikolai\ Tarkhanov},
       DOI = {10.1016/j.matpur.2011.09.001},
       URL = {https://doi.org/10.1016/j.matpur.2011.09.001},
}

@article {fgk,
    AUTHOR = {Fragal\`a, Ilaria and Gazzola, Filippo and Kawohl, Bernd},
     TITLE = {Overdetermined problems with possibly degenerate ellipticity,
              a geometric approach},
   JOURNAL = {Math. Z.},
  FJOURNAL = {Mathematische Zeitschrift},
    VOLUME = {254},
      YEAR = {2006},
    NUMBER = {1},
     PAGES = {117--132},
      ISSN = {0025-5874,1432-1823},
   MRCLASS = {35J60 (35B50 35J25 35J70 35N10)},
  MRNUMBER = {2232009},
MRREVIEWER = {Luisa\ Moschini},
       DOI = {10.1007/s00209-006-0937-7},
       URL = {https://doi.org/10.1007/s00209-006-0937-7},
}

@article {gl,
    AUTHOR = {Garofalo, Nicola and Lewis, John L.},
     TITLE = {A symmetry result related to some overdetermined boundary
              value problems},
   JOURNAL = {Amer. J. Math.},
  FJOURNAL = {American Journal of Mathematics},
    VOLUME = {111},
      YEAR = {1989},
    NUMBER = {1},
     PAGES = {9--33},
      ISSN = {0002-9327,1080-6377},
   MRCLASS = {35N10 (35J60)},
  MRNUMBER = {980297},
MRREVIEWER = {V.\ S.\ Rabinovich},
       DOI = {10.2307/2374477},
       URL = {https://doi.org/10.2307/2374477},
}

@book {li,
    AUTHOR = {Li, Peter},
     TITLE = {Geometric analysis},
    SERIES = {Cambridge Studies in Advanced Mathematics},
    VOLUME = {134},
 PUBLISHER = {Cambridge University Press, Cambridge},
      YEAR = {2012},
     PAGES = {x+406},
      ISBN = {978-1-107-02064-1},
   MRCLASS = {58-02 (35P15 53C21 58J32 58J35)},
  MRNUMBER = {2962229},
MRREVIEWER = {Fr\'ed\'eric\ Robert},
       DOI = {10.1017/CBO9781139105798},
       URL = {https://doi.org/10.1017/CBO9781139105798},
}

@article {gnn,
    AUTHOR = {Gidas, Basilis and Ni, Wei Ming and Nirenberg, Louis},
     TITLE = {Symmetry and related properties via the maximum principle},
   JOURNAL = {Comm. Math. Phys.},
  FJOURNAL = {Communications in Mathematical Physics},
    VOLUME = {68},
      YEAR = {1979},
    NUMBER = {3},
     PAGES = {209--243},
      ISSN = {0010-3616,1432-0916},
   MRCLASS = {35J25 (35B50)},
  MRNUMBER = {544879},
MRREVIEWER = {\`E.\ M.\ Saak},
       URL = {http://projecteuclid.org/euclid.cmp/1103905359},
}

@article {hp,
    AUTHOR = {Henrot, Antoine and Philippin, G\'erard A.},
     TITLE = {Some overdetermined boundary value problems with elliptical
              free boundaries},
   JOURNAL = {SIAM J. Math. Anal.},
  FJOURNAL = {SIAM Journal on Mathematical Analysis},
    VOLUME = {29},
      YEAR = {1998},
    NUMBER = {2},
     PAGES = {309--320},
      ISSN = {0036-1410,1095-7154},
   MRCLASS = {35N10 (30E25 35R35)},
  MRNUMBER = {1616562},
       DOI = {10.1137/S0036141096307217},
       URL = {https://doi.org/10.1137/S0036141096307217},
}

@article {kp98,
    AUTHOR = {Kumaresan, S. and Prajapat, Jyotshana},
     TITLE = {Serrin's result for hyperbolic space and sphere},
   JOURNAL = {Duke Math. J.},
  FJOURNAL = {Duke Mathematical Journal},
    VOLUME = {91},
      YEAR = {1998},
    NUMBER = {1},
     PAGES = {17--28},
      ISSN = {0012-7094,1547-7398},
   MRCLASS = {35J65 (35B05)},
  MRNUMBER = {1487977},
MRREVIEWER = {V.\ Raghavendra},
       DOI = {10.1215/S0012-7094-98-09102-5},
       URL = {https://doi.org/10.1215/S0012-7094-98-09102-5},
}

@article {mmp,
    AUTHOR = {Magnanini, Rolando and Molinarolo, Riccardo and Poggesi,
              Giorgio},
     TITLE = {A general integral identity with applications to a reverse
              {S}errin problem},
   JOURNAL = {J. Geom. Anal.},
  FJOURNAL = {Journal of Geometric Analysis},
    VOLUME = {34},
      YEAR = {2024},
    NUMBER = {8},
     PAGES = {Paper No. 244, 26},
      ISSN = {1050-6926,1559-002X},
   MRCLASS = {35N25 (35A23 35B35 35J05 35M12)},
  MRNUMBER = {4751743},
MRREVIEWER = {Jiabin\ Yin},
       DOI = {10.1007/s12220-024-01693-8},
       URL = {https://doi.org/10.1007/s12220-024-01693-8},
}

@article {mp2020,
    AUTHOR = {Magnanini, Rolando and Poggesi, Giorgio},
     TITLE = {Serrin's problem and {A}lexandrov's soap bubble theorem:
              enhanced stability via integral identities},
   JOURNAL = {Indiana Univ. Math. J.},
  FJOURNAL = {Indiana University Mathematics Journal},
    VOLUME = {69},
      YEAR = {2020},
    NUMBER = {4},
     PAGES = {1181--1205},
      ISSN = {0022-2518,1943-5258},
   MRCLASS = {53A10 (58E12)},
  MRNUMBER = {4124125},
MRREVIEWER = {Jo\~ao\ Lucas Marques Barbosa},
       DOI = {10.1512/iumj.2020.69.7925},
       URL = {https://doi.org/10.1512/iumj.2020.69.7925},
}

@article {mp23,
    AUTHOR = {Magnanini, Rolando and Poggesi, Giorgio},
     TITLE = {Interpolating estimates with applications to some quantitative
              symmetry results},
   JOURNAL = {Math. Eng.},
  FJOURNAL = {Mathematics in Engineering},
    VOLUME = {5},
      YEAR = {2023},
    NUMBER = {1},
     PAGES = {Paper No. 002, 21},
      ISSN = {2640-3501},
   MRCLASS = {49Q05 (35J65 53C42)},
  MRNUMBER = {4370324},
       DOI = {10.3934/mine.2023002},
       URL = {https://doi.org/10.3934/mine.2023002},
}

@article {nt,
    AUTHOR = {Nitsch, Carlo and Trombetti, Cristina},
     TITLE = {The classical overdetermined {S}errin problem},
   JOURNAL = {Complex Var. Elliptic Equ.},
  FJOURNAL = {Complex Variables and Elliptic Equations. An International
              Journal},
    VOLUME = {63},
      YEAR = {2018},
    NUMBER = {7-8},
     PAGES = {1107--1122},
      ISSN = {1747-6933,1747-6941},
   MRCLASS = {35N25 (26D15 35B50 35J05)},
  MRNUMBER = {3802818},
       DOI = {10.1080/17476933.2017.1410798},
       URL = {https://doi.org/10.1080/17476933.2017.1410798},
}

@article {p,
    AUTHOR = {Payne, Lawrence E.},
     TITLE = {Some remarks on maximum principles},
   JOURNAL = {J. Analyse Math.},
  FJOURNAL = {Journal d'Analyse Math\'ematique},
    VOLUME = {30},
      YEAR = {1976},
     PAGES = {421--433},
      ISSN = {0021-7670,1565-8538},
   MRCLASS = {35J65},
  MRNUMBER = {454338},
MRREVIEWER = {H.\ F.\ Weinberger},
       DOI = {10.1007/BF02786729},
       URL = {https://doi.org/10.1007/BF02786729},
}

@article {re77,
    AUTHOR = {Reilly, Robert C.},
     TITLE = {Applications of the {H}essian operator in a {R}iemannian
              manifold},
   JOURNAL = {Indiana Univ. Math. J.},
  FJOURNAL = {Indiana University Mathematics Journal},
    VOLUME = {26},
      YEAR = {1977},
    NUMBER = {3},
     PAGES = {459--472},
      ISSN = {0022-2518,1943-5258},
   MRCLASS = {53C40},
  MRNUMBER = {474149},
MRREVIEWER = {N.\ J.\ Hicks},
       DOI = {10.1512/iumj.1977.26.26036},
       URL = {https://doi.org/10.1512/iumj.1977.26.26036},
}

@article {re82,
    AUTHOR = {Reilly, Robert C.},
     TITLE = {Mean curvature, the {L}aplacian, and soap bubbles},
   JOURNAL = {Amer. Math. Monthly},
  FJOURNAL = {American Mathematical Monthly},
    VOLUME = {89},
      YEAR = {1982},
    NUMBER = {3},
     PAGES = {180--188, 197--198},
      ISSN = {0002-9890,1930-0972},
   MRCLASS = {53A10 (49F10 53C45)},
  MRNUMBER = {645791},
MRREVIEWER = {Hansklaus\ Rummler},
       DOI = {10.2307/2320201},
       URL = {https://doi.org/10.2307/2320201},
}

@article {ro,
    AUTHOR = {Ros, Antonio},
     TITLE = {Compact hypersurfaces with constant higher order mean
              curvatures},
   JOURNAL = {Rev. Mat. Iberoamericana},
  FJOURNAL = {Revista Matem\'atica Iberoamericana},
    VOLUME = {3},
      YEAR = {1987},
    NUMBER = {3-4},
     PAGES = {447--453},
      ISSN = {0213-2230},
   MRCLASS = {53C42 (53C45)},
  MRNUMBER = {996826},
MRREVIEWER = {Franki\ Dillen},
       DOI = {10.4171/RMI/58},
       URL = {https://doi.org/10.4171/RMI/58},
}

@article {ss,
    AUTHOR = {Silvestre, Luis and Sirakov, Boyan},
     TITLE = {Overdetermined problems for fully nonlinear elliptic
              equations},
   JOURNAL = {Calc. Var. Partial Differential Equations},
  FJOURNAL = {Calculus of Variations and Partial Differential Equations},
    VOLUME = {54},
      YEAR = {2015},
    NUMBER = {1},
     PAGES = {989--1007},
      ISSN = {0944-2669,1432-0835},
   MRCLASS = {35N25 (35B06 35B50 35J60)},
  MRNUMBER = {3385189},
MRREVIEWER = {Antonio\ Di Teodoro},
       DOI = {10.1007/s00526-014-0814-x},
       URL = {https://doi.org/10.1007/s00526-014-0814-x},
}

@article {w,
    AUTHOR = {Weinberger, Hans F.},
     TITLE = {Remark on the preceding paper of {S}errin},
   JOURNAL = {Arch. Rational Mech. Anal.},
  FJOURNAL = {Archive for Rational Mechanics and Analysis},
    VOLUME = {43},
      YEAR = {1971},
     PAGES = {319--320},
      ISSN = {0003-9527},
   MRCLASS = {31B20},
  MRNUMBER = {333221},
MRREVIEWER = {Gottfried\ Anger},
       DOI = {10.1007/BF00250469},
       URL = {https://doi.org/10.1007/BF00250469},
}

@incollection {ma2017,
    AUTHOR = {Magnanini, Rolando},
     TITLE = {Alexandrov, {S}errin, {W}einberger, {R}eilly: simmetry and
              stability by integral identities},
 BOOKTITLE = {Bruno {P}ini {M}athematical {A}nalysis {S}eminar 2017},
    SERIES = {Bruno Pini Math. Anal. Semin.},
    VOLUME = {8},
     PAGES = {121--141},
 PUBLISHER = {Univ. Bologna, Alma Mater Stud., Bologna},
      YEAR = {2017},
   MRCLASS = {35N25 (35-02 35A23 35B35 53A10)},
  MRNUMBER = {3893584},
}

@article {b2013,
    AUTHOR = {Brendle, Simon},
     TITLE = {Constant mean curvature surfaces in warped product manifolds},
   JOURNAL = {Publ. Math. Inst. Hautes \'Etudes Sci.},
  FJOURNAL = {Publications Math\'ematiques. Institut de Hautes \'Etudes
              Scientifiques},
    VOLUME = {117},
      YEAR = {2013},
     PAGES = {247--269},
      ISSN = {0073-8301,1618-1913},
   MRCLASS = {53A10 (53C45)},
  MRNUMBER = {3090261},
MRREVIEWER = {Andrew\ Bucki},
       DOI = {10.1007/s10240-012-0047-5},
       URL = {https://doi.org/10.1007/s10240-012-0047-5},
}

@article {fp2022,
    AUTHOR = {Fogagnolo, Mattia and Pinamonti, Andrea},
     TITLE = {New integral estimates in substatic {R}iemannian manifolds and
              the {A}lexandrov theorem},
   JOURNAL = {J. Math. Pures Appl. (9)},
  FJOURNAL = {Journal de Math\'ematiques Pures et Appliqu\'ees. Neuvi\`eme
              S\'erie},
    VOLUME = {163},
      YEAR = {2022},
     PAGES = {299--317},
      ISSN = {0021-7824,1776-3371},
   MRCLASS = {49Q10 (49J40 53C21 53C24 58J32 83C20)},
  MRNUMBER = {4438902},
       DOI = {10.1016/j.matpur.2022.05.007},
       URL = {https://doi.org/10.1016/j.matpur.2022.05.007},
}

@article {crv2021,
    AUTHOR = {Ciraolo, Giulio and Roncoroni, Alberto and Vezzoni, Luigi},
     TITLE = {Quantitative stability for hypersurfaces with almost constant
              curvature in space forms},
   JOURNAL = {Ann. Mat. Pura Appl. (4)},
  FJOURNAL = {Annali di Matematica Pura ed Applicata. Series IV},
    VOLUME = {200},
      YEAR = {2021},
    NUMBER = {5},
     PAGES = {2043--2083},
      ISSN = {0373-3114,1618-1891},
   MRCLASS = {53C20 (35B50 53C21 53C24 53C42)},
  MRNUMBER = {4285109},
MRREVIEWER = {Yijun\ He},
       DOI = {10.1007/s10231-021-01069-7},
       URL = {https://doi.org/10.1007/s10231-021-01069-7},
}

@article {hlmg2009,
    AUTHOR = {He, Yijun and Li, Haizhong and Ma, Hui and Ge, Jianquan},
     TITLE = {Compact embedded hypersurfaces with constant higher order
              anisotropic mean curvatures},
   JOURNAL = {Indiana Univ. Math. J.},
  FJOURNAL = {Indiana University Mathematics Journal},
    VOLUME = {58},
      YEAR = {2009},
    NUMBER = {2},
     PAGES = {853--868},
      ISSN = {0022-2518,1943-5258},
   MRCLASS = {53C42},
  MRNUMBER = {2514391},
MRREVIEWER = {C\'esar\ Rosales},
       DOI = {10.1512/iumj.2009.58.3515},
       URL = {https://doi.org/10.1512/iumj.2009.58.3515},
}

@article {m2005,
    AUTHOR = {Morgan, Frank},
     TITLE = {Planar {W}ulff shape is unique equilibrium},
   JOURNAL = {Proc. Amer. Math. Soc.},
  FJOURNAL = {Proceedings of the American Mathematical Society},
    VOLUME = {133},
      YEAR = {2005},
    NUMBER = {3},
     PAGES = {809--813},
      ISSN = {0002-9939,1088-6826},
   MRCLASS = {49Q05 (74E15)},
  MRNUMBER = {2113931},
MRREVIEWER = {Dorin\ Bucur},
       DOI = {10.1090/S0002-9939-04-07661-0},
       URL = {https://doi.org/10.1090/S0002-9939-04-07661-0},
}

@article {jwxz2023,
    AUTHOR = {Jia, Xiaohan and Wang, Guofang and Xia, Chao and Zhang, Xuwen},
     TITLE = {Alexandrov's theorem for anisotropic capillary hypersurfaces
              in the half-space},
   JOURNAL = {Arch. Ration. Mech. Anal.},
  FJOURNAL = {Archive for Rational Mechanics and Analysis},
    VOLUME = {247},
      YEAR = {2023},
    NUMBER = {2},
     PAGES = {Paper No. 25, 19},
      ISSN = {0003-9527,1432-0673},
   MRCLASS = {53A10 (49Q15 49Q20)},
  MRNUMBER = {4562813},
MRREVIEWER = {Yijun\ He},
       DOI = {10.1007/s00205-023-01861-0},
       URL = {https://doi.org/10.1007/s00205-023-01861-0},
}

@article {dks2020,
    AUTHOR = {De Rosa, Antonio and Kolasi\'nski, S\l awomir and Santilli,
              Mario},
     TITLE = {Uniqueness of critical points of the anisotropic isoperimetric
              problem for finite perimeter sets},
   JOURNAL = {Arch. Ration. Mech. Anal.},
  FJOURNAL = {Archive for Rational Mechanics and Analysis},
    VOLUME = {238},
      YEAR = {2020},
    NUMBER = {3},
     PAGES = {1157--1198},
      ISSN = {0003-9527,1432-0673},
   MRCLASS = {49Q20 (49Q15)},
  MRNUMBER = {4160798},
MRREVIEWER = {David\ Tewodrose},
       DOI = {10.1007/s00205-020-01562-y},
       URL = {https://doi.org/10.1007/s00205-020-01562-y},
}

@article {frs2024,
    AUTHOR = {Freitas, Allan and Roncoroni, Alberto and Santos, M\'arcio},
     TITLE = {A note on {S}errin's type problem on {R}iemannian manifolds},
   JOURNAL = {J. Geom. Anal.},
  FJOURNAL = {Journal of Geometric Analysis},
    VOLUME = {34},
      YEAR = {2024},
    NUMBER = {7},
     PAGES = {Paper No. 200, 17},
      ISSN = {1050-6926,1559-002X},
   MRCLASS = {35R01 (35B50 35N25 53C24 58J05)},
  MRNUMBER = {4737327},
       DOI = {10.1007/s12220-024-01650-5},
       URL = {https://doi.org/10.1007/s12220-024-01650-5},
}

@incollection {mr1991,
    AUTHOR = {Montiel, Sebasti\'an and Ros, Antonio},
     TITLE = {Compact hypersurfaces: the {A}lexandrov theorem for higher
              order mean curvatures},
 BOOKTITLE = {Differential geometry},
    SERIES = {Pitman Monogr. Surveys Pure Appl. Math.},
    VOLUME = {52},
     PAGES = {279--296},
 PUBLISHER = {Longman Sci. Tech., Harlow},
      YEAR = {1991},
      ISBN = {0-582-05590-3},
   MRCLASS = {53C40 (53C42)},
  MRNUMBER = {1173047},
MRREVIEWER = {R.\ H.\ Bowman},
}

@article {cms2016,
    AUTHOR = {Ciraolo, Giulio and Magnanini, Rolando and Sakaguchi, Shigeru},
     TITLE = {Solutions of elliptic equations with a level surface parallel
              to the boundary: stability of the radial configuration},
   JOURNAL = {J. Anal. Math.},
  FJOURNAL = {Journal d'Analyse Math\'ematique},
    VOLUME = {128},
      YEAR = {2016},
     PAGES = {337--353},
      ISSN = {0021-7670,1565-8538},
   MRCLASS = {35J25 (35B06 35B09 35B35 35B50 35B65 35J05)},
  MRNUMBER = {3481178},
MRREVIEWER = {Antonio\ Vitolo},
       DOI = {10.1007/s11854-016-0011-2},
       URL = {https://doi.org/10.1007/s11854-016-0011-2},
}

@article {km2017,
    AUTHOR = {Krummel, Brian and Maggi, Francesco},
     TITLE = {Isoperimetry with upper mean curvature bounds and sharp
              stability estimates},
   JOURNAL = {Calc. Var. Partial Differential Equations},
  FJOURNAL = {Calculus of Variations and Partial Differential Equations},
    VOLUME = {56},
      YEAR = {2017},
    NUMBER = {2},
     PAGES = {Paper No. 53, 43},
      ISSN = {0944-2669,1432-0835},
   MRCLASS = {49Q20 (49K40 49Q10 53C42 58E35)},
  MRNUMBER = {3627438},
MRREVIEWER = {Hojoo\ Lee},
       DOI = {10.1007/s00526-017-1139-3},
       URL = {https://doi.org/10.1007/s00526-017-1139-3},
}

@misc{fz2026pre,
      title={Sharp stability of Alexandrov's theorem for $C^1$ domains in the small-excess regime}, 
      author={Alessio Figalli and Yi Ru-Ya Zhang},
      year={2026},
      eprint={2606.13335},
      archivePrefix={arXiv},
      primaryClass={math.DG},
      url={https://arxiv.org/abs/2606.13335}, 
      howpublished = {Preprint, \doi{https://doi.org/10.48550/arXiv.2606.13335}}
}

@article {Lou2008,
    AUTHOR = {Lou, Hongwei},
     TITLE = {On singular sets of local solutions to {$p$}-{L}aplace
              equations},
   JOURNAL = {Chinese Ann. Math. Ser. B},
  FJOURNAL = {Chinese Annals of Mathematics. Series B},
    VOLUME = {29},
      YEAR = {2008},
    NUMBER = {5},
     PAGES = {521--530},
      ISSN = {0252-9599,1860-6261},
   MRCLASS = {35J60 (35A20 35J70 49J45)},
  MRNUMBER = {2447484},
MRREVIEWER = {Gabriella\ Bogn\'ar},
       DOI = {10.1007/s11401-007-0312-y},
       URL = {https://doi.org/10.1007/s11401-007-0312-y},
}

@article {cnt2024,
    AUTHOR = {Celentano, Antonio and Nitsch, Carlo and Trombetti, Cristina},
     TITLE = {On a {S}errin type overdetermined problem},
   JOURNAL = {Milan J. Math.},
  FJOURNAL = {Milan Journal of Mathematics},
    VOLUME = {92},
      YEAR = {2024},
    NUMBER = {2},
     PAGES = {579--590},
      ISSN = {1424-9286,1424-9294},
   MRCLASS = {35B06 (35J57 35N25)},
  MRNUMBER = {4836151},
MRREVIEWER = {Antonio\ Greco},
       DOI = {10.1007/s00032-024-00405-9},
       URL = {https://doi.org/10.1007/s00032-024-00405-9},
}

@misc{crist2026,
      title={Symmetry-breaking and local stability of a two-phase eigenvalue problem in optimal insulation}, 
      author={Emanuele Cristoforoni and Federico Villone},
      year={2026},
      eprint={2607.07218},
      archivePrefix={arXiv},
      primaryClass={math.AP},
      url={https://arxiv.org/abs/2607.07218}, 
      howpublished = {Preprint, \doi{https://doi.org/10.48550/arXiv.2607.07218}}
      
}

@article {agbomaz2025,
    AUTHOR = {Agostiniani, Virginia and Borghini, Stefano and Mazzieri,
              Lorenzo},
     TITLE = {On the {S}errin problem for ring-shaped domains},
   JOURNAL = {J. Eur. Math. Soc. (JEMS)},
  FJOURNAL = {Journal of the European Mathematical Society (JEMS)},
    VOLUME = {27},
      YEAR = {2025},
    NUMBER = {7},
     PAGES = {2705--2749},
      ISSN = {1435-9855,1435-9863},
   MRCLASS = {35N25 (35B06 35R35 53C21)},
  MRNUMBER = {4910594},
MRREVIEWER = {Giuseppe\ Scianna},
       DOI = {10.4171/jems/1422},
       URL = {https://doi.org/10.4171/jems/1422},
}

\end{document}